\documentclass[a4paper,10pt]{amsart}
\usepackage{amsmath}
\usepackage{amssymb}
\usepackage{amsthm}
\usepackage{shuffle}
\usepackage{hyperref}
\usepackage{here}
\usepackage{mathtools}
\usepackage{stmaryrd}

\numberwithin{equation}{section}

\newtheorem{theorem}{Theorem}
\newtheorem*{theorem*}{Theorem}
\newtheorem{lemma}{Lemma}
\newtheorem*{lemma*}{Lemma}

\newtheorem*{corollary*}{Corollary}

\newtheorem*{propositon*}{Proposition}

\theoremstyle{definition}
\newtheorem{definition}{Definition}

\newtheorem*{definition*}{Definition}

\newtheorem*{remark*}{Remark}

\newtheorem*{example*}{Example}

\DeclareMathOperator{\Li}{Li}

\allowdisplaybreaks[1]

\begin{document}
\title[An explicit parity theorem for multiple polylogarithms]{An explicit parity theorem for multiple polylogarithms}
\author{RYOTA UMEZAWA}
\address{National Institute of Technology, Toyota College, 
2-1 Eiseicho, Toyota, Aichi, Japan, 471-8525.}
\email{umezawa@toyota-ct.ac.jp}
\date{}
\maketitle

\begin{abstract}
In 2017, E. Panzer proved the parity theorem for multiple polylogarithms. While his proof is simple and constructive, it does not yield a general explicit formula. Inspired by M. Hirose's recent work on the parity theorem for multiple zeta values, this paper presents an explicit formula for the parity theorem for multiple polylogarithms.
\end{abstract}
\section{Introduction}\label{se:intro}
For $\mathbf{k} = (k_{1},\dots,k_{d}) \in \mathbb{Z}_{>0}^{d}$ with $k_{d} >1$, the values 
\[\zeta(\mathbf{k}) = \zeta(k_{1},\dots,k_{d}) \coloneqq \sum_{0 < m_{1} < \cdots < m_{d}} \frac{1}{m_{1}^{k_{1}} \cdots m_{d}^{k_{d}}}\]
are called multiple zeta values. Additionally, $|\mathbf{k}| \coloneqq k_{1}+\dots+k_{d}$ is called the weight, and $\textrm{dep}(\mathbf{k}) \coloneqq d$ is called the depth. We denote the index of weight $0$ and depth $0$ by $\emptyset$, and consider $\zeta(\emptyset) = 1$. There is a well-known theorem concerning multiple zeta values, called the parity theorem or the parity result, which has been proven by several authors, such as  \cite{Br}, \cite{IKZ}, \cite{J}, \cite{M} and \cite{T}.
\begin{theorem}[Parity theorem for multiple zeta values]
The multiple zeta value $\zeta(k_{1},\dots,k_{d})$ can be expressed as a $\mathbb{Q}$-linear combination of products of multiple zeta values of depth at most $d-1$, provided that its depth and weight have opposite parity.
\end{theorem}
There are also several analogues and extensions of the theorem, for example, \cite{K}, \cite{KM}, \cite{KNS}, \cite{O}, \cite{P}, \cite{T2},  \cite{T3}. In particular, Panzer \cite{P} extended the parity theorem to multiple polylogarithms. The multiple polylogarithms are defined by
\begin{equation}\label{eq:Lise}
\Li_{\mathbf{k}}(\boldsymbol{z}) = \Li_{k_{1},\dots,k_{d}}(z_{1},\dots,z_{d}) \coloneqq \sum_{0 < m_{1} < \cdots < m_{d}} \frac{z_{1}^{m_{1}} \cdots z_{d}^{m_{d}}}{m_{1}^{k_{1}} \cdots m_{d}^{k_{d}}}
\end{equation}
for $\mathbf{k} = (k_{1},\dots,k_{d}) \in \mathbb{Z}_{> 0}^{d}$ and $\boldsymbol{z} = (z_{1},\dots,z_{d}) \in \mathbb{C}^{d}$. We consider $\Li_{\emptyset}(\emptyset) = 1$. Let $z_{i,j}$ denote a consecutive product  $z_{i}\cdots z_{j}$. The series on the right-hand side of \eqref{eq:Lise} converges absolutely if $|z_{i,d}|<1$ for all $i \in \{1,\dots,d\}$ and in that region, it is equal to the integral expression
\begin{equation} \label{eq:Liint}
\Li_{\mathbf{k}}(\boldsymbol{z}) = (-1)^{d}\int_{0}^{1}\omega_{1/z_{1,d}}\omega_{0}^{k_{1}-1}\omega_{1/z_{2,d}}\omega_{0}^{k_{2}-1}\cdots \omega_{1/z_{d,d}}\omega_{0}^{k_{d}-1},
\end{equation}
where we consider the integral of the right-hand side as $0$ when there exists $i$ such that $z_{i}=0$. The symbol $\omega_{a}$ represents the differential form $\omega_{a} = dt/(t-a)$, and the above integral represents the iterated integral along the straight path from $0$ to $1$:
\[\int_{0}^{1}\omega_{a_{1}}\omega_{a_{2}}\cdots\omega_{a_{n}} = \int_{0}^{1} \frac{dt_{n}}{t_{n}-a_{n}} \cdots \int_{0}^{t_{3}} \frac{dt_{2}}{t_{2}-a_{2}} \int_{0}^{t_{2}} \frac{dt_{1}}{t_{1}-a_{1}}.\]
For a subset $S$ of $\mathbb{C}$, we define
\[\mathcal{D}^{d}(S) \coloneqq \mathbb{C}^{d} \setminus \bigcup_{1 \le i \le j \le d} \{(z_{1},\dots,z_{d}) \in  \mathbb{C}^{d} \mid z_{i, j} \in S\}\]
and
\[\tilde{\mathcal{D}}^{d}(S) \coloneqq \mathbb{C}^{d} \setminus \bigcup_{1 \le i \le d} \{(z_{1},\dots,z_{d}) \in  \mathbb{C}^{d} \mid z_{i, d} \in S\}.\]
The integral expression \eqref{eq:Liint} gives an analytic continuation to the set $\tilde{\mathcal{D}}^{d}(\mathbb{R}_{\ge 1})$. Moreover, the integral of the right-hand side of \eqref{eq:Liint} converges absolutely in $\tilde{\mathcal{D}}^{d}(\mathbb{R}_{> 1})$ except when $(k_{d}, z_{d}) = (1,1)$, and continuous there. When $(k_{d}, z_{d}) = (1,1)$, the right-hand side of \eqref{eq:Liint} diverges, but we can give it a meaning through regularization, which will be explained in Section \ref{se:reg}. Let $\Li^{*}_{\mathbf{k}}(\boldsymbol{z})$ denote the stuffle regularized values, and $\Li^{\shuffle}_{\mathbf{k}}(\boldsymbol{z})$ denote the shuffle regularized values. Note that in many papers, $\Li^{*}_{\mathbf{k}}(\boldsymbol{z})$ and $\Li^{\shuffle}_{\mathbf{k}}(\boldsymbol{z})$ are used with different meanings from ours. For convenience, $\Li_{\mathbf{k}}(\boldsymbol{z})$ is sometimes denoted as $\Li^{\emptyset}_{\mathbf{k}}(\boldsymbol{z})$, but typically the $\emptyset$ is omitted.

To state Pnazer's theorem, we introduce some notation, which is also used in the main theorems. For $\boldsymbol{z} = (z_{1},\dots,z_{d}) \in \mathbb{C}^{d}$, we write
\[1/\boldsymbol{z} \coloneqq \left(\frac{1}{z_{1}},\dots, \frac{1}{z_{d}}\right).\]
Additionally, we define 
\[\mathcal{B}_{l}(z) \coloneqq \frac{(2\pi i)^{l}}{l!}B_{l}\left(\frac{1}{2} + \frac{\log(-z)}{2\pi i}\right),\]
where $i$ denotes the imaginary unit, $B_{l}(x)$ denotes the $l$-th Bernoulli polynomial, and $\log(z)$ denotes the principal branch of the logarithm, with a branch cut along $(-\infty, 0]$.

Panzer's theorem can then be stated as follows.
\begin{theorem}[Parity theorem for multiple polylogarithms (Panzer \cite{P})]
For $\mathbf{k} \in \mathbb{Z}_{>0}^{d}$ and $\boldsymbol{z} \in \mathcal{D}^{d}(\mathbb{R}_{\ge 0})$,
\[\Li_{\mathbf{k}}(\boldsymbol{z}) - (-1)^{d+|\mathbf{k}|}\Li^{}_{\mathbf{k}}(1/\boldsymbol{z})\]
can be expressed as a sum of the form
\[\prod_{i=1}^{d}\mathcal{B}_{l_{i}}(z_{i,d}) \prod_{j=1}^{s} \Li_{\mathbf{k}^{(j)}}(\boldsymbol{z}^{(j)})\]
with integer coefficients, where $s, l_{1},\dots,l_{d} \in \mathbb{Z}_{\ge 0}$ and $\mathbf{k}^{(j)} \in \mathbb{Z}_{> 0}^{d_{j}}$ satisfy
\[l_{1}+\cdots+l_{d}+|\mathbf{k}^{(1)}|+\cdots+|\mathbf{k}^{(s)}|=|\mathbf{k}| \mathrm{\quad and \quad} d_{1}+\cdots+d_{s} < d,\]
and $\boldsymbol{z}^{(j)} \in \mathbb{C}^{d_{j}}$ are allowed only if they can be written as
\[(\boldsymbol{z}^{(1)},\dots,\boldsymbol{z}^{(s)}) = (z_{a_{1},b_{1}}, z_{a_{2},b_{2}}, \dots, z_{a_{h}, b_{h}})\]
with $1 \le a_{1} \le b_{1} < a_{2} \le b_{2} < \cdots < a_{h} \le b_{h} \le d$ and $h=d_{1}+\cdots+d_{s}$.
\end{theorem}
It should be noted that Panzer \cite{P} furthermore discusses the treatment of $z_{i}\to1$ using shuffle regularization. Panzer's proof is simple and constructive, but it does not provide an explicit formula for general depth.

Recently, Hirose \cite{H} provided a simple explicit formula for the parity theorem for multiple zeta values via a theory of multitangent functions by Bouillot \cite{B}. To state Hirose's theorem, we introduce some notation. Since the main theorems use more general notation, we introduce inclusive notation to avoid redundancy.
\begin{definition} For $\bullet \in \{\emptyset, *, \shuffle\}$, $a \in \mathbb{Z}_{\ge 0}$, $\mathbf{k}=(k_{1}, \dots, k_{d}) \in \mathbb{Z}_{>0}^{d}$ and $\boldsymbol{z}=(z_{1},\dots,z_{d}) \in \mathcal{D}^{d}(\mathbb{R}_{> 1})$, we define
\begin{align*}
\Li^{\star, \bullet}_{\mathbf{k}}(\boldsymbol{z}) &\coloneqq \sum_{(\mathbf{k}';\boldsymbol{z}') \preceq (\mathbf{k};\boldsymbol{z})} \Li^{\bullet}_{\mathbf{k}'}(\boldsymbol{z}'), \\
\Li_{\mathbf{k}}^{a, \bullet}(\boldsymbol{z}) &\coloneqq (-1)^{a}\sum_{\substack{l_{1}+\cdots+l_{d}=a\\ l_{1},\dots,l_{d} \ge 0}} \left(\prod_{i=1}^{d}\binom{k_{i}+l_{i}-1}{l_{i}}\right)\Li^{\bullet}_{k_{1}+l_{1},\dots,k_{d}+l_{d}}(z_{1},\dots,z_{d}),\\
\tilde{\Li}_{\mathbf{k}}^{a, \bullet}(\boldsymbol{z})
&\coloneqq (-1)^{a}\sum_{\substack{l_{1}+\cdots+l_{d}=a\\ l_{1},\dots,l_{d} \ge 0}} \left(\prod_{i=1}^{d}\binom{k_{i}+l_{i}-1}{l_{i}}\right)\\
&\qquad \times \sum_{\substack{(\boldsymbol{z}^{(1)}, \dots, \boldsymbol{z}^{(s)}) = (z_{1}, \dots, z_{d})\\ (\mathbf{k}^{(1)}, \dots, \mathbf{k}^{(s)}) = (k_{1}+l_{1}, \dots, k_{d}+l_{d})}}(-1)^{d+s}\Li^{\star, \bullet}_{\mathbf{k}^{(1)}}(\boldsymbol{z}^{(1)}) \cdots \Li^{\star, \bullet}_{\mathbf{k}^{(s)}}(\boldsymbol{z}^{(s)}).
\end{align*}
Here, $(\mathbf{k}';\boldsymbol{z}') \preceq (\mathbf{k};\boldsymbol{z})$ indicates that the pair of $\mathbf{k}'$ and $\boldsymbol{z}'$ can be written as
\[\mathbf{k}' = (k_{1}\boxempty_{1} \cdots \boxempty_{d-1}k_{d})\quad \text{and} \quad \boldsymbol{z}' = (z_{1}\boxempty'_{1} \cdots \boxempty'_{d-1} z_{d}),\]
where each pair $(\boxempty_{i}, \boxempty'_{i})$ is either $(\text{``}, \text{''}, \text{``}, \text{''})$ or $(\text{``} + \text{''}, \text{``} \times \text{''})$. The inner sum on the right-hand side of the third equation is over all partitions of $(k_{1}+l_{1}, \dots, k_{d}+l_{d})$ and $(z_{1}, \dots, z_{d})$ into an arbitrary number $s$ of parts, denoted by $\mathbf{k}^{(1)}, \dots, \mathbf{k}^{(s)}$ and $\boldsymbol{z}^{(1)}, \dots, \boldsymbol{z}^{(s)}$, such that $\textrm{dep}(\mathbf{k}^{(i)}) = \textrm{dep}(\boldsymbol{z}^{(i)})>0$ for all $i \in \{1, \dots, s\}$. We consider $\Li^{\star, \bullet}_{\emptyset}(\emptyset)$ and $\Li^{a, \bullet}_{\emptyset}(\emptyset)$ $(a=0)$ as $1$, and $\Li^{a, \bullet}_{\emptyset}(\emptyset)$ $(a>0)$ as $0$. In addition, we condider $\tilde{\Li}_{\mathbf{k}}^{a, \bullet}(\boldsymbol{z})$ $(a<0)$ as $0$.
\end{definition}
Let $\{k\}^{n}$ denote $n$ repetitions of $k$. We use the same notation for multiple zeta values $\zeta(\mathbf{k}) = \Li_{\mathbf{k}}(\{1\}^{\mathrm{dep}(\mathbf{k})})$ as we do for multiple polylogarithms. In other words, $\zeta^{\bullet,\bullet}(\mathbf{k}) \coloneqq \Li^{\bullet,\bullet}_{\mathbf{k}}(\{1\}^{\mathrm{dep}(\mathbf{k})})$. For $\mathbf{k} = (k_{1},\dots,k_{d})$, we define 
\begin{align*}
\overleftarrow{\mathbf{k}} \coloneqq (k_{d}, \dots, k_{1}) \quad \text{and} \quad \mathbf{k}_{i,j} \coloneqq (k_{i}, \dots, k_{j})
\end{align*}
for $1 \le i \le j \le d$. We consider $\mathbf{k}_{i,j}=\emptyset$ if $i>j$. Similarly, we write $\boldsymbol{z}_{i,j} \coloneqq (z_{i},\dots,z_{j})$.

Then, the formula proved by Hirose is as follows.
\begin{theorem}[Hirose {\cite[Theorem 7]{H}}] \label{th:Hirose} For $\mathbf{k} = (k_{1},\dots,k_{d}) \in \mathbb{Z}^{d}_{> 0}$, we have
\begin{align}\label{eq:Hirose}
&(-1)^{d} \zeta^{\star, *}(\mathbf{k}) - (-1)^{|\mathbf{k}|}\zeta^{*}(\mathbf{k})\\ \nonumber
&=-\sum_{m=0}^{d-1}(-1)^{m}\zeta^{\star, *}(\mathbf{k}_{1,m})\, \delta^{\mathbf{k}_{m+1,d}}\\ \nonumber
&\quad + \sum_{0 \le m < n \le d} \sum_{\substack{a+b+2l=k_{n} \\ a,b, l \ge 0}}  (-1)^{m+|\mathbf{k}_{n+1,d}|+b+l} \frac{(2\pi)^{2l}}{(2l)!}B_{2l}\\ \nonumber
&\qquad \times \zeta^{\star, *}(\mathbf{k}_{1,m})\, \zeta^{a, *}(\overleftarrow{\mathbf{k}_{m+1,n-1}})\, \zeta^{b, *}(\mathbf{k}_{n+1,d}),
\end{align}
where $B_{2l}$ denotes the $2l$-th Bernoulli number, and
\[\delta^{k_{1},\dots,k_{d}} \coloneqq
\begin{cases}
\frac{(-1)^{d/2}\pi^{d}}{d!} & (k_{1},\dots,k_{d}) = (\{1\}^{d}) \mathrm{\ and\ } d \in 2\mathbb{Z},\\
0 &\mathrm{otherwise.}
\end{cases}\]
\end{theorem}
Equation \eqref{eq:Hirose} is an explicit formula for the parity theorem for multiple zeta values because the left-hand side of \eqref{eq:Hirose} can be written as 
\[(-1)^{d}2\zeta^{*}(k_{1},\dots,k_{d}) + (-1)^{d}\sum_{\substack{\boxempty_{i} = \text{``}, \text{''} \text{\ or\ } \text{``} + \text{''} \\ \text{for\ } i = 1, \dots, d-1 \\ \text{not\ all\ } \boxempty_{i} = \text{``}, \text{''}}} \zeta^{*}_{}(k_{1}\boxempty_{1} \cdots \boxempty_{d-1}k_{d})\]
when $k_{1}+\dots+k_{d} \not\equiv d \pmod 2$.

The present paper aims to extend Hirose's formula to multiple polylogarithms and to provide an explicit formula for Panzer's theorem. More precisely, we prove Theorem \ref{th:main} and its extension, Theorem \ref{th:main2}, which are the main theorems of this paper.

\begin{theorem} \label{th:main}
For $\mathbf{k} = (k_{1},\dots, k_{d}) \in \mathbb{Z}_{>0}^{d}$ and $\boldsymbol{z} = (z_{1},\dots, z_{d}) \in \mathcal{D}^{d}(\mathbb{R}_{\ge 0})$,
we have
\begin{align}\label{eq:main}
& (-1)^{d}\Li^{\star}_{\mathbf{k}}(\boldsymbol{z}) - (-1)^{|\mathbf{k}|}\Li_{\mathbf{k}}(1/\boldsymbol{z}) \\ \nonumber
&= \sum_{0 \le m < n \le d} \sum_{\substack{a+b+ l=k_{n} \\ a,b, l \ge 0}}  (-1)^{m+|\mathbf{k}_{n+1,d}|+b} \mathcal{B}_{l}(z_{m+1,d})\\ \nonumber
&\qquad \times \Li^{\star}_{\mathbf{k}_{1,m}}(\boldsymbol{z}_{1,m})\, \tilde{\Li}^{a}_{\mathbf{k}_{m+1,n-1}}(\boldsymbol{z}_{m+1,n-1})\Li^{b}_{\mathbf{k}_{n+1,d}}(1/\boldsymbol{z}_{n+1,d}).
\end{align}
\end{theorem}
Note that this theorem differs slightly from Panzer's in that it involves multiple polylogarithms with arguments that are the reciprocals of $z_{i}$. However, we can prove Panzer's theorem by repeatedly applying equation \eqref{eq:main} to $\Li^{b}_{\mathbf{k}_{n+1,d}}(1/\boldsymbol{z}_{n+1,d})$.

It is known from Zhao \cite{Z} that $\Li_{\mathbf{k}}(\boldsymbol{z})$ can be analytically continued as a multi-valued holomorphic function on $\mathcal{D}^{d}(\{1\})$. Therefore, by taking the appropriate branch, equation \eqref{eq:main} holds in $\mathcal{D}^{d}(\{1\})$. However, since we want to consider multiple polylogarithms as single-valued functions, we restrict $\boldsymbol{z}$ to $\mathcal{D}^{d}(\mathbb{R}_{\ge 0})$. We are particularly interested in multiple zeta values and colored multiple zeta values (multiple polylogarithms where all $z_{i}$ are $N$-th roots of unity), but these do not satisfy the assumptions of Theorem \ref{th:main}. We thus solve this issue by computing the limit and present the following theorem.

\begin{theorem}\label{th:main2}
For $\mathbf{k} = (k_{1},\dots, k_{d}) \in \mathbb{Z}_{>0}^{d}$, $\boldsymbol{z} = (z_{1},\dots, z_{d}) \in \mathcal{D}^{d}(\mathbb{R}_{\ge 0}\setminus\{1\})$ and $\bullet \in \{*, \shuffle\}$, 
we have
\begin{align}\label{eq:main2}
& (-1)^{d}\Li^{\star, \bullet}_{\mathbf{k}}(\boldsymbol{z}) - (-1)^{|\mathbf{k}|}\Li^{\bullet}_{\mathbf{k}}(1/\boldsymbol{z}) \\ \nonumber
&=-\sum_{m=0}^{d-1} (-1)^{m} \Li^{\star, \bullet}_{\mathbf{k}_{1,m}}(\boldsymbol{z}_{1,m})\, \delta_{\mathbf{k}_{m+1,d}}(\boldsymbol{z}_{m+1,d})\\ \nonumber
&\quad + \sum_{0 \le m < n \le d} \sum_{\substack{a+b+ l=k_{n} \\ a,b, l \ge 0}} (-1)^{m+|\mathbf{k}_{n+1,d}|+b} \mathcal{B}_{l}(z_{m+1,d})\\ \nonumber
&\qquad \times \Li^{\star, \bullet}_{\mathbf{k}_{1,m}}(\boldsymbol{z}_{1,m})\, \tilde{\Li}^{a, \bullet}_{\mathbf{k}_{m+1,n-1}}(\boldsymbol{z}_{m+1,n-1})\Li^{b, \bullet}_{\mathbf{k}_{n+1,d}}(1/\boldsymbol{z}_{n+1,d}),
\end{align}
where
\[\delta_{k_{1},\dots,k_{d}}(z_{1}, \dots, z_{d}) \coloneqq \begin{cases}
\frac{(\log(-z_{d}))^{d}}{d!} & (k_{1},\dots,k_{d}) = (z_{1}, \dots, z_{d}) = (\{1\}^{d}), \\
0 &\mathrm{otherwise.}
\end{cases}\]
\end{theorem}
The value of $\log(-1)$ can be either $\pi i$ or $-\pi i$, and in either case, equation $\eqref{eq:main2}$ holds. Therefore, we can obtain two different expressions depending on the choice, when equation $\eqref{eq:main2}$ contains $\log(-1)$.

From the identity known as the antipode identity,
\begin{align*}
\Li_{k_{d}, \dots k_{1}}^{*}(z_{d},\dots,z_{1})&=\sum_{i=1}^{d}(-1)^{i+1}\Li^{\star, *}_{k_{1},\dots,k_{i}}(z_{1},\dots,z_{i})\Li^{*}_{k_{d}, \dots k_{i+1}}(z_{d},\dots,z_{i+1})\\
&=\sum_{\substack{(\boldsymbol{z}^{(1)}, \dots, \boldsymbol{z}^{(s)}) = (z_{1},\dots,z_{d})\\ (\mathbf{k}^{(1)}, \dots, \mathbf{k}^{(s)}) = (k_{1},\dots,k_{d})}} (-1)^{d+s}\Li^{\star, *}_{\mathbf{k}^{(1)}}(\boldsymbol{z}^{(1)}) \cdots \Li^{\star, *}_{\mathbf{k}^{(s)}}(\boldsymbol{z}^{(s)}),
\end{align*}
it follows that
\[\tilde{\Li}_{k_{1},\dots,k_{d}}^{a, *}(z_{1},\dots,z_{d}) = \Li_{k_{d},\dots,k_{1}}^{a,*}(z_{d},\dots,z_{1}).\]
Therefore, Theorem \ref{th:Hirose} can be considered as a special case of Theorem \ref{th:main2}.

This paper is organized as follows. Section \ref{sec:Proof} provides the proof of Theorem \ref{th:main}, and Section \ref{sec:Proof2} provides the proof of Theorem \ref{th:main2}.

\section{Proof of Theorem \ref{th:main}} \label{sec:Proof}
This section provides a proof of Theorem \ref{th:main}. The proof is based on the method used by Panzer \cite{P}.
\subsection{Outline of proof of Theorem \ref{th:main}} \label{sec:outline} We prove Theorem \ref{th:main} using differentiation with respect to $z_{1}$ and induction on the weight of $\mathbf{k}$, as in Panzer \cite{P}. When the weight is $1$, Theorem \ref{th:main} states that
\[-\Li_{1}(z_{1}) + \Li_{1}(1/z_{1}) = \mathcal{B}_{1}(z_{1}),\]
namely,
\[\log(1-z_{1}) - \log(1-1/z_{1}) = \log(-z_{1}),\]
which also holds when using the principal branch of the logarithm.

We denote the left-hand side of \eqref{eq:main} as $P_{\mathbf{k}}(\boldsymbol{z})$ and the right-hand side as $Q_{\mathbf{k}}(\boldsymbol{z})$. We aim to prove $P_{\mathbf{k}}(\boldsymbol{z}) = Q_{\mathbf{k}}(\boldsymbol{z})$ using the induction hypothesis. For this purpose, we use the following two lemmas. Lemma \ref{le:diff} will be proved later.
\begin{lemma}\label{le:diff} We have
\begin{align}\label{eq:diffP}
&\partial_{z_{1}} P_{\mathbf{k}}(\boldsymbol{z}) = \begin{dcases}
\frac{P_{{}_{-}\mathbf{k}}(\boldsymbol{z})}{z_{1}} & (k_{1} > 1),\\
\frac{P_{\mathbf{k}_{2,d}}({}_{\times}\boldsymbol{z}_{1,d}) - P_{\mathbf{k}_{2,d}}(\boldsymbol{z}_{2,d})}{1-z_{1}} + (-1)^{|\mathbf{k}_{2,d}|}\frac{\Li_{\mathbf{k}_{2,d}}(1/\boldsymbol{z}_{2,d})}{z_{1}} & (k_{1}=1)
\end{dcases}
\end{align}
and 
\begin{align}\label{eq:diffQ}
&\partial_{z_{1}} {Q}_{\mathbf{k}}(\boldsymbol{z}) = \begin{dcases}
\frac{Q_{{}_{-}\mathbf{k}}(\boldsymbol{z})}{z_{1}} & (k_{1} > 1),\\
\frac{Q_{\mathbf{k}_{2,d}}({}_{\times}\boldsymbol{z}_{1,d}) - Q_{\mathbf{k}_{2,d}}(\boldsymbol{z}_{2,d})}{1-z_{1}} + (-1)^{|\mathbf{k}_{2,d}|}\frac{\Li_{\mathbf{k}_{2,d}}(1/\boldsymbol{z}_{2,d})}{z_{1}} & (k_{1}=1),
\end{dcases}
\end{align}
where ${}_{-}\mathbf{k} \coloneqq (k_{1}-1,k_{2},\dots,k_{d})$ and ${}_{\times}\boldsymbol{z}_{1,d} \coloneqq (z_{1}z_{2},z_{3},\dots,z_{d})$.
\end{lemma}
\begin{lemma}[{Panzer \cite[Lemma 2.3]{P}}]\label{le:Panzer} We have
\[\lim_{z_{1} \to 0}\left(\Li_{\mathbf{k}}(1/\boldsymbol{z})+\sum_{\substack{b+ l=k_{1} \\ b, l \ge 0}}(-1)^{k_{1}+b}\mathcal{B}_{l}(z_{1,d})\Li^{b}_{\mathbf{k}_{2,d}}(1/\boldsymbol{z}_{2,d}) \right) = 0.\]
\end{lemma}
By Lemma \ref{le:diff} and the induction hypothesis, we have $\partial_{z_{1}} P_{\mathbf{k}}(\boldsymbol{z}) = \partial_{z_{1}} Q_{\mathbf{k}}(\boldsymbol{z})$. On the other hand, Lemma \ref{le:Panzer} implies
\[\lim_{z_{1} \to 0} \left(P_{\mathbf{k}}(\boldsymbol{z})-Q_{\mathbf{k}}(\boldsymbol{z})\right) = 0.\]
Note that $z_{1}$ must tend to $0$ while keeping $\boldsymbol{z} \in \mathcal{D}^{d}(\mathbb{R}_{\ge 0})$, that is, without any of $z_{1,i}$ $(1 \le i \le d)$ crossing $\mathbb{R}_{\ge 0}$. This can be achieved from any point in $\mathcal{D}^{d}(\mathbb{R}_{\ge 0})$. Therefore, we have $P_{\mathbf{k}}(\boldsymbol{z}) = Q_{\mathbf{k}}(\boldsymbol{z})$. Thus, all that remains is to prove Lemma \ref{le:diff}. 

\subsection{Proof of Lemma \ref{le:diff}}
This subsection proves Lemma \ref{le:diff}. The following differential formulas can be obtained through straightforward calculation:
\begin{align*}
&\partial_{z_{1}} \mathcal{B}_{l}(z_{1,d}) =\begin{dcases} \frac{1}{z_{1}}\mathcal{B}_{l-1}(z_{1,d}) &(l>0),\\
0 & (l=0),
\end{dcases}\\
&\partial_{z_{1}} \Li_{\mathbf{k}}(\boldsymbol{z}) =\begin{dcases}
\frac{1}{z_{1}}\Li_{{}_{-}\mathbf{k}}(\boldsymbol{z})&(k_{1}>1), \\
\frac{1}{1-z_{1}}\Li_{\mathbf{k}_{2,d}}(\boldsymbol{z}_{2,d})-\frac{1}{z_{1}(1-z_{1})}\Li_{\mathbf{k}_{2,d}}({}_{\times}\boldsymbol{z}_{1,d})&(k_{1}=1),
\end{dcases}\\
&\partial_{z_{1}} \Li^{\star}_{\mathbf{k}}(\boldsymbol{z}) =\begin{dcases}
\frac{1}{z_{1}}\Li^{\star}_{{}_{-}\mathbf{k}}(\boldsymbol{z})&(k_{1}>1),\\
\frac{1}{1-z_{1}}(\Li^{\star}_{\mathbf{k}_{2,d}}(\boldsymbol{z}_{2,d})-\Li^{\star}_{\mathbf{k}_{2,d}}({}_{\times}\boldsymbol{z}_{1,d}))&(k_{1}=1),
\end{dcases}\\
&\partial_{z_{1}} \tilde{\Li}_{\mathbf{k}}^{a}(\boldsymbol{z}) =\begin{dcases}
\frac{1}{z_{1}}(\tilde{\Li}_{{}_{-}\mathbf{k}}^{a}(\boldsymbol{z}) - \tilde{\Li}_{\mathbf{k}}^{a-1}(\boldsymbol{z}))&(k_{1}>1), \\
\frac{1}{1-z_{1}}\tilde{\Li}_{\mathbf{k}_{2,d}}^{a}({}_{\times}\boldsymbol{z}_{1,d})-\frac{1}{z_{1}}\tilde{\Li}_{\mathbf{k}}^{a-1}(\boldsymbol{z})&(k_{1}=1).
\end{dcases}
\end{align*}
Here, when $d=1$, we consider $\Li_{\mathbf{k}_{2,d}}({}_{\times}\boldsymbol{z}_{1,d})$, $\Li^{\star}_{\mathbf{k}_{2,d}}({}_{\times}\boldsymbol{z}_{1,d})$ and $\tilde{\Li}_{\mathbf{k}_{2,d}}^{a}({}_{\times}\boldsymbol{z}_{1,d})$ $(a > 0)$ as $0$, and $\tilde{\Li}_{\mathbf{k}_{2,d}}^{a}({}_{\times}\boldsymbol{z}_{1,d})$ $(a = 0)$ as $1$.

Equation \eqref{eq:diffP} can be immediately verified from the above formulas. We now proceed to prove equation \eqref{eq:diffQ}. If we set
\begin{align*}
&R^{n}_{\mathbf{k}_{m+1,d}}(\boldsymbol{z}_{m+1,d})\\
&\coloneqq \sum_{\substack{a+b+ l=k_{n} \\ a,b, l \ge 0}}  (-1)^{b} \mathcal{B}_{l}(z_{m+1,d})\, \tilde{\Li}^{a}_{\mathbf{k}_{m+1,n-1}}(\boldsymbol{z}_{m+1,n-1})\Li^{b}_{\mathbf{k}_{n+1,d}}(1/\boldsymbol{z}_{n+1,d}),
\end{align*}
then $Q_{\mathbf{k}}(\boldsymbol{z})$ can be represented as
\[Q_{\mathbf{k}}(\boldsymbol{z}) = \sum_{0 \le m < n \le d} (-1)^{m+|\mathbf{k}_{n+1,d}|} \Li^{\star}_{\mathbf{k}_{1,m}}(\boldsymbol{z}_{1,m})\, R^{n}_{\mathbf{k}_{m+1,d}}(\boldsymbol{z}_{m+1,d}).\]
The factor $R^{n}_{\mathbf{k}_{m+1,d}}(\boldsymbol{z}_{m+1,d})$ depends on $z_{1}$ only when $m=0$, in which case the following lemma gives its derivative.
\begin{lemma}\label{le:diifR}We have
\[\partial_{z_{1}} R^{n}_{\mathbf{k}_{1,d}}(\boldsymbol{z}_{1,d}) = \begin{dcases}
\frac{1}{z_{1}} R^{n}_{{}_{-}\mathbf{k}_{1,d}}(\boldsymbol{z}_{1,d}) & (k_{1} >1),\\
\frac{1}{z_{1}} \Li^{}_{\mathbf{k}_{2,d}}(1/\boldsymbol{z}_{2,d}) & (k_{1} = 1, n=1),\\
\frac{1}{1-z_{1}} R^{n}_{\mathbf{k}_{2,d}}({}_{\times}\boldsymbol{z}_{1,d}) & (k_{1} = 1, n>1).
\end{dcases}\]
\end{lemma}
\begin{proof}
When $n=1$, since
\[R^{1}_{\mathbf{k}_{1,d}}(\boldsymbol{z}_{1,d}) = \sum_{\substack{b+ l=k_{1} \\ b, l \ge 0}} (-1)^{b} \mathcal{B}_{l}(z_{1,d})\Li^{b}_{\mathbf{k}_{2,d}}(1/\boldsymbol{z}_{2,d}),\]
we have
\[\partial_{z_{1}} R^{1}_{\mathbf{k}_{1,d}}(\boldsymbol{z}_{1,d}) = \begin{dcases}\frac{1}{z_{1}}R^{1}_{{}_{-}\mathbf{k}_{1,d}}(\boldsymbol{z}_{1,d}) & (k_{1} >1),\\
\frac{1}{z_{1}} \Li^{}_{\mathbf{k}_{2,d}}(1/\boldsymbol{z}_{2,d}) & (k_{1} = 1).
\end{dcases}\]
In what follows, we assume $n>1$. If $k_{1} > 1$, we have
\begin{align*}
&\partial_{z_{1}} \left(\mathcal{B}_{l}(z_{1,d})\, \tilde{\Li}^{a}_{\mathbf{k}_{1,n-1}}(\boldsymbol{z}_{1,n-1}) \right)\\
& = 
\begin{dcases}
\frac{1}{z_{1}}  \mathcal{B}_{l}(z_{1,d}) \left(\tilde{\Li}^{a}_{{}_{-}\mathbf{k}_{1,n-1}}(\boldsymbol{z}_{1,n-1}) - \tilde{\Li}^{a-1}_{\mathbf{k}_{1,n-1}}(\boldsymbol{z}_{1,n-1})\right) & (l=0),\\
\begin{aligned}
&\frac{1}{z_{1}} \mathcal{B}_{l-1}(z_{1,d})\, \tilde{\Li}^{a}_{\mathbf{k}_{1,n-1}}(\boldsymbol{z}_{1,n-1}) \\
&\quad +\frac{1}{z_{1}}  \mathcal{B}_{l}(z_{1,d})\left(\tilde{\Li}^{a}_{{}_{-}\mathbf{k}_{1,n-1}}(\boldsymbol{z}_{1,n-1}) - \tilde{\Li}^{a-1}_{\mathbf{k}_{1,n-1}}(\boldsymbol{z}_{1,n-1})\right)
\end{aligned} & (l> 0).
\end{dcases}
\end{align*}
Because some terms cancel out, we obtain
\begin{align*}
\partial_{z_{1}} R^{n}_{\mathbf{k}_{1,d}}(\boldsymbol{z}_{1,d})
&= \frac{1}{z_{1}} \sum_{\substack{a+b+ l=k_{n} \\ a,b, l \ge 0}}  (-1)^{b} \mathcal{B}_{l}(z_{1,d})\, \tilde{\Li}^{a}_{{}_{-}\mathbf{k}_{1,n-1}}(\boldsymbol{z}_{1,n-1})\Li^{b}_{\mathbf{k}_{n+1,d}}(1/\boldsymbol{z}_{n+1,d})\\
&=\frac{1}{z_{1}} R^{n}_{{}_{-}\mathbf{k}_{1,d}}(\boldsymbol{z}_{1,d}).
\end{align*}
On the other hand, if $k_{1} = 1$, we have
\begin{align*}
&\partial_{z_{1}} \left(\mathcal{B}_{l}(z_{1,d})\, \tilde{\Li}^{a}_{\mathbf{k}_{1,n-1}}(\boldsymbol{z}_{1,n-1}) \right)\\
& = 
\begin{dcases}
\frac{1}{1-z_{1}}\mathcal{B}_{l}(z_{1,d})\, \tilde{\Li}_{\mathbf{k}_{2,n-1}}^{a}({}_{\times}\boldsymbol{z}_{1,n-1})-\frac{1}{z_{1}}\mathcal{B}_{l}(z_{1,d})\, \tilde{\Li}_{\mathbf{k}_{1,n-1}}^{a-1}(\boldsymbol{z}_{1,n-1}) & (l=0),\\
\begin{aligned}
&\frac{1}{z_{1}} \mathcal{B}_{l-1}(z_{1,d})\, \tilde{\Li}_{\mathbf{k}_{1,n-1}}^{a}(\boldsymbol{z}_{1,n-1}) \\
& \quad +\frac{1}{1-z_{1}}\mathcal{B}_{l}(z_{1,d})\, \tilde{\Li}_{\mathbf{k}_{2,n-1}}^{a}({}_{\times}\boldsymbol{z}_{1,n-1})-\frac{1}{z_{1}}\mathcal{B}_{l}(z_{1,d})\, \tilde{\Li}_{\mathbf{k}_{1,n-1}}^{a-1}(\boldsymbol{z}_{1,n-1})
\end{aligned} & (l>0).
\end{dcases}
\end{align*}
Because some terms again cancel out, we obtain
\begin{align*}
\partial_{z_{1}} R^{n}_{\mathbf{k}_{1,d}}(\boldsymbol{z}_{1,d})
&= \frac{1}{1-z_{1}} \sum_{\substack{a+b+ l=k_{n} \\ a,b, l \ge 0}}  (-1)^{b} \mathcal{B}_{l}(z_{1,d})\, \tilde{\Li}_{\mathbf{k}_{2,n-1}}^{a}({}_{\times}\boldsymbol{z}_{1,n-1}) \Li^{b}_{\mathbf{k}_{n+1,d}}(1/\boldsymbol{z}_{n+1,d})\\
&=\frac{1}{1-z_{1}} R^{n}_{\mathbf{k}_{2,d}}({}_{\times}\boldsymbol{z}_{1,d}).
\end{align*}
Thus, the lemma has been verified in all cases.
\end{proof}

We now return to the proof of equation \eqref{eq:diffQ}. We decompose $Q_{\mathbf{k}}(\boldsymbol{z})$ into the following two parts:
\begin{align*}
Q_{\mathbf{k}}(\boldsymbol{z}) &= \sum_{0 < n \le d} (-1)^{|\mathbf{k}_{n+1,d}|} R^{n}_{\mathbf{k}_{1,d}}(\boldsymbol{z}_{1,d})\\
&\quad  + \sum_{0 < m < n \le d} (-1)^{m+|\mathbf{k}_{n+1,d}|} \Li^{\star}_{\mathbf{k}_{1,m}}(\boldsymbol{z}_{1,m})\, R^{n}_{\mathbf{k}_{m+1,d}}(\boldsymbol{z}_{m+1,d}),
\end{align*}
and compute its derivative. If $k_{1} > 1$, then we have
\begin{align*}
\partial_{z_{1}} Q_{\mathbf{k}}(\boldsymbol{z}) &= \frac{1}{z_{1}} \sum_{0 < n \le d} (-1)^{|\mathbf{k}_{n+1,d}|} R^{n}_{{}_{-}\mathbf{k}_{1,d}}(\boldsymbol{z}_{1,d})\\
&\quad  + \frac{1}{z_{1}} \sum_{0 < m < n \le d} (-1)^{m+|\mathbf{k}_{n+1,d}|} \Li^{\star}_{{}_{-}\mathbf{k}_{1,m}}(\boldsymbol{z}_{1,m})\, R^{n}_{\mathbf{k}_{m+1,d}}(\boldsymbol{z}_{m+1,d})\\
&=\frac{1}{z_{1}}  Q_{{}_{-}\mathbf{k}}(\boldsymbol{z})
\end{align*}
and if $k_{1} = 1$, then we have
\begin{align*}
\partial_{z_{1}} Q_{\mathbf{k}}(\boldsymbol{z}) &= (-1)^{|\mathbf{k}_{2,d}|} \frac{1}{z_{1}}  \Li^{}_{\mathbf{k}_{2,d}}(1/\boldsymbol{z}_{2,d})\\
&\quad + \frac{1}{1-z_{1}} \sum_{1 < n \le d} (-1)^{|\mathbf{k}_{n+1,d}|} R^{n}_{\mathbf{k}_{2,d}}({}_{\times}\boldsymbol{z}_{1,d})\\
&\quad + \frac{1}{1-z_{1}} \sum_{0 < m < n \le d} (-1)^{m+|\mathbf{k}_{n+1,d}|} \Li^{\star}_{\mathbf{k}_{2,m}}(\boldsymbol{z}_{2,m})\, R^{n}_{\mathbf{k}_{m+1,d}}(\boldsymbol{z}_{m+1,d})\\
&\quad - \frac{1}{1-z_{1}} \sum_{1 < m < n \le d} (-1)^{m+|\mathbf{k}_{n+1,d}|} \Li^{\star}_{\mathbf{k}_{2,m}}({}_{\times}\boldsymbol{z}_{1,m})\, R^{n}_{\mathbf{k}_{m+1,d}}(\boldsymbol{z}_{m+1,d})\\
&= \frac{Q_{\mathbf{k}_{2,d}}({}_{\times}\boldsymbol{z}_{1,d}) - Q_{\mathbf{k}_{2,d}}(\boldsymbol{z}_{2,d})}{1-z_{1}} + (-1)^{|\mathbf{k}_{2,d}|}\frac{\Li_{\mathbf{k}_{2,d}}(1/\boldsymbol{z}_{2,d})}{z_{1}}.
\end{align*}
This completes the proof of Lemma \ref{le:diff}, and thus completes the proof of Theorem \ref{th:main}.

\section{Proof of Theorem \ref{th:main2}}\label{sec:Proof2}
This section provides the proof of Theorem \ref{th:main2}. The method of proof is to compute the limit in equation \eqref{eq:main} using regularization techniques.
\subsection{Outline of proof of Theorem \ref{th:main2}}\label{subsec:outline2}
We denote the left-hand side of \eqref{eq:main} by $P_{\mathbf{k}}(\boldsymbol{z})$ and the right-hand side by $Q_{\mathbf{k}}(\boldsymbol{z})$, as in Section \ref{sec:outline}. To prove Theorem \ref{th:main2}, we fix $\boldsymbol{z} = (z_{1}, \dots ,z_{d}) \in \mathcal{D}^{d}(\mathbb{R}_{\ge 0}\setminus\{1\})$, and compute the limit of both sides of 
\begin{equation}
P_{\mathbf{k}}(\boldsymbol{v}) = Q_{\mathbf{k}}(\boldsymbol{v}) \label{eq:P=Q}
\end{equation}
as $\boldsymbol{v}=(v_{1},\dots,v_{d})$ tends to $\boldsymbol{z}$ from within $\mathcal{D}^{d}(\mathbb{R}_{\ge 0})$, taking the limit sequentially in the order $v_{1} \to z_{1}, \dots, v_{d} \to z_{d}$. To facilitate the formulation, we choose paths $\gamma_{i}$ tending to each $z_{i}$, such that none of the consecutive products of variables crosses $\mathbb{R}_{\ge 0}$. More precisely, a tuple of paths $\boldsymbol{z} = (\gamma_{1}, \dots, \gamma_{d})$, where each $\gamma_{i}:[0,1] \to \mathbb{C} \setminus (\mathbb{R}_{\ge 0}\setminus\{1\})$ with $\gamma_{i}(1) = z_{i}$,
must be chosen such that, for any $i \in \{1, \dots, d\}$ and any $t \in [0, 1)$, 
\[(\gamma_{1}(1), \dots, \gamma_{i-1}(1), \gamma_{i}(t), \gamma_{i+1}(0), \dots, \gamma_{d}(0))\]
belongs to
\[\mathbb{C}^{d} \setminus \bigcup_{1 \le j \le k \le d} \{ (z_{1},\dots,z_{d}) \in  \mathbb{C}^{d} \mid z_{j,k} \in S_{k} \},\]
where
\[S_{k} = \begin{cases}\mathbb{R}_{\ge 0} \setminus \{1\} &\text{if}\quad k < i, \\ \mathbb{R}_{\ge 0} &\text{if}\quad i \le k.\end{cases}\]
For fixed $\boldsymbol{z} = (z_{1}, \cdots, z_{d}) \in \mathcal{D}^{d}(\mathbb{R}_{\ge 0}\setminus\{1\})$, a tuple of paths $\boldsymbol{\gamma} = (\gamma_{1}, \dots, \gamma_{d})$ tending to $\boldsymbol{z}$ and satisfying the above condition, and $i \in \{1, \dots, d\}$, we define
\[D_{\boldsymbol{z}, \boldsymbol{\gamma}, i} = \{(\gamma_{1}(1), \dots, \gamma_{i-1}(1), \gamma_{i}(t), \gamma_{i+1}(0), \dots, \gamma_{d}(0)) \mid t \in [0,1)\}.\]
Thus, when writing $(\boldsymbol{z}_{1,i-1},\boldsymbol{v}_{i,d}) \in D_{\boldsymbol{z}, \boldsymbol{\gamma}, i}$, only $v_{i}$ is a variable, and the other entries are fixed.

When $z_{m+1,d} = 1$, we consider the values of $\log(-z_{m+1,d})$ to be either $\pi i$ or $-\pi i$, depending on the path $\gamma_{d}$. More precisely, as $z_{m+1,d-1}v_{d} \to z_{m+1,d} = 1$ from above the real axis, $\log(-z_{m+1,d}) = -\pi i$, while from below the real axis, $\log(-z_{m+1,d}) = \pi i$.

Recall that $\Li_{\mathbf{k}}\left(\boldsymbol{z}\right)$ is continuous on $\tilde{\mathcal{D}}^{d}(\mathbb{R}_{> 1})$ except when $(k_{d}, z_{d}) = (1,1)$, in which case it diverges. This implies that both $\Li_{\mathbf{k}}\left(\boldsymbol{z}\right)$ and $\Li_{\mathbf{k}}(1/\boldsymbol{z})$ converge to limits that do not depend on whether  $z_{i,d} \to 1$ from above or below the real axis, except in the case $k_{d}=1$ and $z_{d} \to 1$, in which they diverge. Thus, as $v_{i} \to z_{i}$, both sides of equation \eqref{eq:P=Q} converge except when $k_{i}=1$ and $z_{i} = 1$. When $k_{i}=1$ and $v_{i} \to z_{i} = 1$ $(1 \le i <d)$, the right-hand side of \eqref{eq:P=Q} involves divergent terms. However, the left-hand side converges, so the divergence of the right-hand side must cancel. We will compute these cancellations explicitly using regularization techniques.
\subsection{Regularization}\label{se:reg}
We use an algebraic setup introduced by Hoffmann \cite{Ho}. The contents of this subsection are based on \cite{AK} and \cite{IKZ}. We define the noncommutative polynomial ring $\mathfrak{H} \coloneqq \mathbb{Q}\left<x, y_{z}\, \mid z \in \mathbb{C} \right>$. Let $\mathfrak{H}^{1}$ be the subspace of $\mathfrak{H}$ that does not contain words starting with $x$, and let $\mathfrak{H}^{0}$ be the subspace of $\mathfrak{H}^{1}$ that does not contain words ending with $y_{1}$. For $\mathbf{k} = (k_{1},\dots,k_{d}) \in \mathbb{Z}_{> 0}^{d}$ and $\boldsymbol{z} = (z_{1},\dots,z_{d}) \in \mathbb{C}^{d}$, we define
\[w_{\mathbf{k};\boldsymbol{z}} \coloneqq y_{z_{1}}x^{k_{1}-1}y_{z_{2}}x^{k_{2}-1} \cdots y_{z_{d}}x^{k_{d}-1} \in \mathfrak{H}^{1}\]
and a $\mathbb{Q}$-linear map $\mathcal{I}:\mathfrak{H}^{1} \to \mathfrak{H}^{1}$ by
\[\mathcal{I}(w_{\mathbf{k};\boldsymbol{z}}) = y_{z_{1,d}}x^{k_{1}-1}y_{z_{2,d}}x^{k_{2}-1} \cdots y_{z_{d,d}}x^{k_{d}-1}.\]
Additionary, we define two $\mathbb{Q}$-linear maps $\mathcal{L}^{*}$ and $\mathcal{L}^{\shuffle}$ $:\mathfrak{H}^{0}\to \mathbb{C}$ as follows: 
\[\mathcal{L}^{*} (w_{\mathbf{k};\boldsymbol{z}}) = \Li_{\mathbf{k}}(\boldsymbol{z}) \quad \text{and} \quad \mathcal{L}^{\shuffle} (\mathcal{I}(w_{\mathbf{k};\boldsymbol{z}})) = \Li_{\mathbf{k}}(\boldsymbol{z}).\]
When $\boldsymbol{z} \notin \tilde{\mathcal{D}}^{d}(\mathbb{R}_{> 1})$, the right-hand sides are not well-defined, however, we do not deal with such  $\boldsymbol{z}$.
Two types of products are defined on $\mathfrak{H}$ and its subspaces $\mathfrak{H}^{1}$ and $\mathfrak{H}^{0}$. The first is the stuffle product $*$, which is defined recursively by $w*1=1 * w=w$ and
\begin{align*}
&y_{z_{1}}x^{n_{1}}w_{1} * y_{z_{2}}x^{n_{2}}w_{2} \\
&= y_{z_{1}}x^{n_{1}}(w_{1}* y_{z_{2}}x^{n_{2}}w_{2})+y_{z_{2}}x^{n_{2}}(y_{z_{1}}x^{n_{1}}w_{1} * w_{2}) + y_{z_{1}z_{2}}x^{n_{1}+n_{2}+1}(w_{1} * w_{2})
\end{align*}
for any words $w, w_{1},w_{2}$ and $n_{1}, n_{2} \in \mathbb{Z}_{\ge 0}$, together with $\mathbb{Q}$-bilinearity. The second is the shuffle product $\shuffle$, which is recursively defined by $w\shuffle1=1\shuffle w=w$ and
\[u_{1}w_{1}\shuffle u_{2}w_{2}=u_{1}(w_{1}\shuffle u_{2}w_{2})+u_{2}(u_{1}w_{1}\shuffle w_{2})\]
for any words $w, w_{1},w_{2}$ and letters $u_{1},u_{2}$, together with $\mathbb{Q}$-bilinearity. The maps $\mathcal{L}^{*}$ and $\mathcal{L}^{\shuffle}$ are homomorphisms with respect to the products $*$ and $\shuffle$, respectively, provided that the expressions involved are well-defined. In other words, the following identities hold when both sides consist only of well-defined multiple polylogarithms:
\[\mathcal{L}^{*} (w_{1} * w_{2}) = \mathcal{L}^{*} (w_{1} )\, \mathcal{L}^{*} (w_{2}) \quad \text{and} \quad \mathcal{L}^{\shuffle} (w_{1} \shuffle w_{2}) = \mathcal{L}^{\shuffle} (w_{1} )\, \mathcal{L}^{\shuffle} (w_{2} ).\]
Let $u^{\bullet n}$ with $\bullet \in \{*, \shuffle \}$ denote
\[u^{\bullet n} \coloneqq \underbrace{u \bullet u \bullet \cdots \bullet u}_{n}.\]
The words $w_{\mathbf{k};\boldsymbol{z}}$ and $\mathcal{I}(w_{\mathbf{k};\boldsymbol{z}})$ can be written uniquely in the form
\[w_{\mathbf{k};\boldsymbol{z}} =w_{0} + w_{1} * y_{1}^{* 1} + \cdots + w_{h} * y_{1}^{* h}\]
and
\[\mathcal{I}(w_{\mathbf{k};\boldsymbol{z}}) =w'_{0} + w'_{1} \shuffle y_{1}^{\shuffle 1} + \cdots + w'_{h'}\shuffle y_{1}^{\shuffle h'},\]
where $w_{0},\dots,w_{h},w'_{0},\dots,w'_{h'} \in \mathfrak{H}^{0}$. Note that if $\boldsymbol{z} \in \tilde{\mathcal{D}}^{d}(\mathbb{R}_{>1})$, both $\mathcal{L}^{*} (w_{i})$ and $\mathcal{L}^{\shuffle} (w'_{i})$ are well-defined. When $\mathcal{I}(w_{\mathbf{k};\boldsymbol{z}})$ is written in the form of $wuy_{1}^{h}$ with a word $w$, a letter $u \neq y_{1}$ and non-negative integer $h$, $w'_{i}$ are explicitly written as
\[w'_{i} = \frac{(-1)^{h-i}}{i!}(w\shuffle y_{1}^{h-i})u. \label{eq:w'_{i} =}\]
Otherwise, since $\mathcal{I}(w_{\mathbf{k};\boldsymbol{z}}) = y_{1}^{d}$, we have
\[w'_{i} = \begin{cases}
0 & (i \neq d), \\
\frac{1}{d!} & (i = d).
\end{cases}\]
We define the polynomials $\mathcal{L}^{*}_{\mathbf{k};\boldsymbol{z}}(T)$ and $\mathcal{L}^{\shuffle}_{\mathbf{k};\boldsymbol{z}}(T)$ by
\begin{align*}
\mathcal{L}^{*}_{\mathbf{k};\boldsymbol{z}}(T) &\coloneqq \mathcal{L}^{*} (w_{0})+\mathcal{L}^{*} (w_{1})T + \cdots + \mathcal{L}^{*} (w_{h})T^{h},\\
\mathcal{L}^{\shuffle}_{\mathbf{k};\boldsymbol{z}}(T) &\coloneqq \mathcal{L}^{\shuffle} (w'_{0})+\mathcal{L}^{\shuffle} (w'_{1})T + \cdots + \mathcal{L}^{\shuffle} (w'_{h'})T^{h'},
\end{align*}
respectively. The two polynomials $\mathcal{L}^{*}_{\mathbf{k};\boldsymbol{z}}(T)$ and $\mathcal{L}^{\shuffle}_{\mathbf{k};\boldsymbol{z}}(T)$ are related by
\[\mathcal{L}^{\shuffle}_{\mathbf{k};\boldsymbol{z}}(T) = \rho \circ \mathcal{L}^{*}_{\mathbf{k};\boldsymbol{z}}(T),\]
where $\rho$: $\mathbb{C}[T] \to \mathbb{C}[T]$ is an invertible $\mathbb{C}$-linear map defined by the generating function:
\[\rho(e^{Tu}) = \Gamma(1+u)e^{(T+\gamma)u}.\]
Here, $\Gamma(s)$ denotes the Gamma function, and $\gamma$ is the Euler constant. For $\boldsymbol{z} \in \tilde{\mathcal{D}}^{d}(\mathbb{R}_{>1})$, where $(z_{d},k_{d}) = (1,1)$ is also allowed, we define
\[\Li^{*}_{\mathbf{k}}\left(\boldsymbol{z}\right) \coloneqq  \mathcal{L}^{*}_{\mathbf{k};\boldsymbol{z}}(0) \quad \text{and} \quad \Li^{\shuffle}_{\mathbf{k}}\left(\boldsymbol{z}\right) \coloneqq \mathcal{L}^{\shuffle}_{\mathbf{k};\boldsymbol{z}}(0).\]
These values are called stuffle regularized values and shuffle regularized values, respectively. Let $h$ be a non-negative integer such that $\mathbf{k}_{d-h+1, d} = \boldsymbol{z}_{d-h+1, d} = (\{1\}^{h})$ and $(k_{d-h}, z_{d-h}) \neq (1,1)$. Then, we have
\begin{equation}
\mathcal{L}^{\shuffle}_{\mathbf{k};\boldsymbol{z}}(T) = \sum_{m=0}^{h} \Li^{\shuffle}_{\mathbf{k}_{1,d-m}}\left(\boldsymbol{z}_{1,d-m}\right) \frac{T^{m}}{m!}. \label{eq:shufflepolynomial}
\end{equation}
We will also use the following notation later:
\[\mathcal{L}^{\star, \shuffle}_{\mathbf{k};\boldsymbol{z}}(T) \coloneqq \sum_{(\mathbf{k}';\boldsymbol{z}') \preceq (\mathbf{k};\boldsymbol{z})} \mathcal{L}^{\shuffle}_{\mathbf{k}';\boldsymbol{z}'}(T).\]

\subsection{Asymptotic behavior as $v_{i} \to z_{i}$}

For $\mathbf{k} \in \mathbb{Z}_{>0}^{d}$ and $(\boldsymbol{z}_{1,d-1}, {v}_{d}) \in D_{\boldsymbol{z}, \boldsymbol{\gamma}, d}$, the word $\mathcal{I}(w_{\mathbf{k}; \boldsymbol{z}_{1,d-1}, {v}_{d}})$ can be written uniquely in the form
\[\mathcal{I}(w_{\mathbf{k}; \boldsymbol{z}_{1,d-1}, {v}_{d}}) = w_{0} + w_{1} \shuffle y_{v_{d}}^{\shuffle 1} + \cdots + w_{h}\shuffle y_{v_{d}}^{\shuffle h},\]
where none of $w_{0}, \dots, w_{h}$ contains a word ending with $y_{v_{d}}$ or $y_{1}$. Using $w_{0}, \dots, w_{h}$ from above, we define
\[\mathcal{L}^{+}_{\mathbf{k}; \boldsymbol{z}_{1,d-1}, {v}_{d}}(T) \coloneqq 
\begin{cases}\mathcal{L}^{\shuffle} (w_{0})+\mathcal{L}^{\shuffle} (w_{1})T + \cdots + \mathcal{L}^{\shuffle} (w_{h})T^{h} & \mathrm{if}\ (k_{d}, z_{d}) = (1, 1),\\
\Li_{\mathbf{k}}(\boldsymbol{z}_{1,d-1}, {v}_{d}) & \mathrm{if}\ z_{d} \neq 1.
\end{cases}\]
In addition, we define
\begin{align*}
\mathcal{L}^{-}_{\mathbf{k};\boldsymbol{z}_{1,d-1}, {v}_{d}}(T) &\coloneqq \mathcal{L}^{+}_{\mathbf{k};\boldsymbol{z}_{1,d-1}, {v}_{d}}(T-\log(-v_{d})),\\
\mathcal{L}^{\star, +}_{\mathbf{k};\boldsymbol{z}_{1,d-1}, {v}_{d}}(T) &\coloneqq \sum_{(\mathbf{k}';\boldsymbol{z}') \preceq (\mathbf{k};\boldsymbol{z}_{1,d-1}, {v}_{d})}\mathcal{L}^{+}_{\mathbf{k}';\boldsymbol{z}'}(T),\\
\mathcal{L}^{\star, -}_{\mathbf{k};\boldsymbol{z}_{1,d-1}, {v}_{d}}(T) &\coloneqq \sum_{(\mathbf{k}';\boldsymbol{z}') \preceq (\mathbf{k};\boldsymbol{z}_{1,d-1}, {v}_{d})}\mathcal{L}^{-}_{\mathbf{k}';\boldsymbol{z}'}(T).
\end{align*}
Then, it follows that
\begin{align*}
\Li_{\mathbf{k}}(\boldsymbol{z}_{1,d-1}, {v}_{d}) &= \mathcal{L}^{+}_{\mathbf{k}; \boldsymbol{z}_{1,d-1}, {v}_{d}}(-\log(1-v_{d})) = \mathcal{L}^{-}_{\mathbf{k};\boldsymbol{z}_{1,d-1}, {v}_{d}}(-\log(1-1/v_{d})),\\
\Li^{\star}_{\mathbf{k}}(\boldsymbol{z}_{1,d-1}, {v}_{d}) &= \mathcal{L}^{\star, +}_{\mathbf{k};\boldsymbol{z}_{1,d-1}, {v}_{d}}(-\log(1-v_{d})) = \mathcal{L}^{\star, -}_{\mathbf{k};\boldsymbol{z}_{1,d-1}, {v}_{d}}(-\log(1-1/v_{d})).
\end{align*}
It is clear from the definition that each coefficient of $\mathcal{L}^{+}_{\mathbf{k}; \boldsymbol{z}_{1,d-1}, {v}_{d}}(T)$ converges to the corresponding coefficient of $\mathcal{L}^{\shuffle}_{\mathbf{k};\boldsymbol{z}}(T)$ as $v_{d} \to z_{d}$.  The order of convergence is as follows.
\begin{lemma}\label{le:polynomialz}
Let $\mathrm{coeff}_{T^{n}}P(T)$ denote the coefficient of $T^{n}$ in the polynomial $P(T)$.
For $(\boldsymbol{z}_{1,d-1}, {v}_{d}) \in D_{\boldsymbol{z}, \boldsymbol{\gamma}, d}$, we have
\[\mathrm{coeff}_{T^{n}} \mathcal{L}^{+}_{\mathbf{k};\boldsymbol{z}_{1,d-1}, {v}_{d}}(T) = \mathrm{coeff}_{T^{n}} \mathcal{L}^{\shuffle}_{\mathbf{k};\boldsymbol{z}}(T) + O\left((z_{d}-v_{d})\log^{d-n}(z_{d}-v_{d})\right)\]
as $v_{d} \to z_{d}$ along the path $\gamma_{d}$.
\end{lemma}
\begin{proof}
We divide the proof into two cases.

\textbf{Case $(k_{d}, z_{d}) = (1, 1)$.} We use induction on the weight of $\mathbf{k}$. We first note that, when $\mathbf{k} = \boldsymbol{z} = (\{1\}^{d})$,
\begin{equation}\mathcal{L}^{+}_{\mathbf{k};\boldsymbol{z}_{1,d-1}, {v}_{d}}(T) = \mathcal{L}^{\shuffle}_{\mathbf{k};\boldsymbol{z}}(T) = \frac{1}{d!}T^{d}. \label{eq:1/dT^d}
\end{equation}

When $|\mathbf{k}| = 1$, the assertion is obvious from equation \eqref{eq:1/dT^d} with $d=1$. We prove the assertion for $\mathbf{k}$ with $k_{d} = 1$ by assuming it holds for all indices of weight less than $|\mathbf{k}|$ and with last entry $1$. If $\mathbf{k} = \boldsymbol{z} = (1, \dots, 1)$, then the assertion is obvious from equation \eqref{eq:1/dT^d}. Otherwise, there exist a positive integer $h$, a word $w$ and a letter $u \notin \{y_{1}, y_{v_{d}}\}$ such that $\mathcal{I}(w_{\mathbf{k};\boldsymbol{z}_{1,d-1}, {v}_{d}}) = wuy_{v_{d}}^{h}$. All that remains is to prove that
\[\mathcal{L}^{\shuffle}\left( (w\shuffle y_{v_{d}}^{h-n})u\right) - \mathcal{L}^{\shuffle}\left( (w\shuffle y_{v_{d}}^{h-n})u\big|_{v_{d} = 1}\right) = O\left((1-v_{d})\log^{d-n}(1-v_{d})\right)\]
for all $ n \in \{0, \dots, h\}$. Note that there are two cases:
\begin{itemize}
  \item $u = x$, i.e., $\mathbf{k}_{d-h+1, d} = \boldsymbol{z}_{d-h+1, d} = (\{1\}^{h})$ and $k_{d-h} \neq 1$,
  \item $u = y_{z_{d-h}v_{d}}$, i.e., $\mathbf{k}_{d-h+1, d} = \boldsymbol{z}_{d-h+1, d} = (\{1\}^{h})$, $k_{d-h} = 1$ and $z_{d-h} \neq 1$.
\end{itemize}
In each case, we obtain
\begin{align*}
&\mathcal{L}^{\shuffle}\left( (w\shuffle y_{v_{d}}^{h-n})u\right)\\
&=\begin{dcases}
\int_{0}^{v_{d}} \Li_{\mathbf{k}_{1,d-h-1},k_{d-h}-1}(\boldsymbol{z}_{1,d-h-1},z_{d-h}t) \frac{\left(\Li_{1}(t)\right)^{h-n}dt}{t} & (u=x),\\
\int_{0}^{v_{d}} \Li_{\mathbf{k}_{1,d-h-1}}(\boldsymbol{z}_{1,d-h-2},z_{d-h-1,d-h}t) \frac{\left(\Li_{1}(t)\right)^{h-n}dt}{1/z_{d-h}-t} & (u = y_{z_{d-h}v_{d}}),
\end{dcases}
\end{align*}
where the integration is taken along the straight path from $0$ to $v_{d}$. Therefore, we have
\begin{align*}
&\mathcal{L}^{\shuffle}\left( (w\shuffle y_{v_{d}}^{h-n})u\right) - \mathcal{L}^{\shuffle}\left( (w\shuffle y_{v_{d}}^{h-n})u\big|_{v_{d} = 1}\right)\\
&=\begin{dcases}
\int_{1}^{v_{d}} \Li_{\mathbf{k}_{1,d-h-1},k_{d-h}-1}(\boldsymbol{z}_{1,d-h-1},z_{d-h}t) \frac{\left(\Li_{1}(t)\right)^{h-n}dt}{t} & (u=x),\\
\int_{1}^{v_{d}} \Li_{\mathbf{k}_{1,d-h-1}}(\boldsymbol{z}_{1,d-h-2},z_{d-h-1,d-h}t) \frac{\left(\Li_{1}(t)\right)^{h-n}dt}{1/z_{d-h}-t} & (u = y_{z_{d-h}v_{d}}).
\end{dcases}
\end{align*}
Note that the path obtained by concatenating the straight path from $1$ to $0$ and the straight path from $0$ to $v_{d}$ can be deformed into the straight path from $1$ to $v_{d}$.
By the induction hypothesis, if $(k_{d-h}-1,z_{d-h}) = (1,1)$, then
\begin{align*}
\Li_{\mathbf{k}_{1,d-h-1},k_{d-h}-1}(\boldsymbol{z}_{1,d-h-1},z_{d-h}t) &= \mathcal{L}^{+}_{\mathbf{k}_{1,d-h-1},1; \boldsymbol{z}_{1,d-h-1},t}(-\log(1-t)) \\
&= O\left(\log^{d-h}(1-t)\right),
\end{align*}
and otherwise it is $O(1)$ as $t \to 1$. Similarly, if $z_{d-h} \neq 1$ and $(k_{d-h-1},z_{d-h-1,d-h}) = (1,1)$, then
\[\Li_{\mathbf{k}_{1,d-h-1}}(\boldsymbol{z}_{1,d-h-2},z_{d-h-1,d-h}t) = O\left(\log^{d-h-1}(1-t)\right),\]
and otherwise it is $O(1)$ as $t \to 1$. Therefore, by the change of variable $t = 1-(1-v_{d})t$, we obtain
\begin{align*}
&\mathcal{L}^{\shuffle}\left( (w\shuffle y_{v_{d}}^{h-n})u\right) - \mathcal{L}^{\shuffle}\left( (w\shuffle y_{v_{d}}^{h-n})u\big|_{v_{d} = 1}\right)\\
&\ll |1-v_{d}|\int_{0}^{1} \left|-\log((1-v_{d})t)\right|^{d-n}dt\\
&\ll \left|(1-v_{d})\log^{d-n}(1-v_{d})\right|.
\end{align*}

\textbf{Case $(k_{d},z_{d}) \neq (1,1)$.} In this case, it follows that
\[\mathcal{L}^{+}_{\mathbf{k};\boldsymbol{z}_{1,d-1}, {v}_{d}}(T) = \Li_{\mathbf{k}}(\boldsymbol{z}_{1,d-1},{v}_{d}) \quad \mathrm{and} \quad \mathcal{L}^{\shuffle}_{\mathbf{k};\boldsymbol{z}}(T) = \Li_{\mathbf{k}}(\boldsymbol{z}).\]
Therefore, we only need to prove that
\begin{align}
\Li_{\mathbf{k}}(\boldsymbol{z}_{1,d-1},{v}_{d}) - \Li_{\mathbf{k}}(\boldsymbol{z}) = O\left((z_{d}-v_{d})\log^{d}(z_{d}-v_{d})\right). \label{eq:v_{d}-z_{d}}
\end{align}
We obtain
\[\Li_{\mathbf{k}}(\boldsymbol{z}_{1,d-1},{v}_{d}) - \Li_{\mathbf{k}}(\boldsymbol{z}) = \begin{dcases}
\int_{z_{d}}^{v_{d}} \Li_{\mathbf{k}_{1,d-1},k_{d}-1}(\boldsymbol{z}_{1,d-1},t)\frac{dt}{t} & (k_{d} \neq 1),\\
\int_{z_{d}}^{v_{d}} \Li_{\mathbf{k}_{1,d-1}}(\boldsymbol{z}_{1,d-2},z_{d-1}t)\frac{dt}{1-t} & (k_{d} = 1).
\end{dcases}\]
By the result for the case $(k_{d}, z_{d}) = (1, 1)$, if $(k_{d}-1,z_{d}) = (1,1)$, then 
\[\Li_{\mathbf{k}_{1,d-1},k_{d}-1}(\boldsymbol{z}_{1,d-1},t) = O\left(\log^{d}(1-t)\right),\]
and if $(k_{d-1},z_{d-1,d}) = (1,1)$, then
\[\Li_{\mathbf{k}_{1,d-1}}(\boldsymbol{z}_{1,d-2},z_{d-1}t) = O\left(\log^{d-1}(1-z_{d-1}t)\right)\]
as $t \to z_{d}$. In all other cases, the integrands are $O(1)$. Therefore, by the change of variable $t = z_{d}-(z_{d}-v_{d})t$, we obtain equation \eqref{eq:v_{d}-z_{d}}.
\end{proof}
Let $(\boldsymbol{z}_{1,i-1},\boldsymbol{v}_{i,d}) \in D_{\boldsymbol{z}, \boldsymbol{\gamma}, i}$ with $z_{i} = 1$. We write $f(v_{i}) \sim g(v_{i})$ if $f(v_{i}) - g(v_{i}) = O((1-v_{i})^{1-\varepsilon})$ as $v_{i} \to z_{i} = 1$ along the path $\gamma_{i}$. In addition, for two polynomials in the indeterminate $T$,
\[P_{v_{i}}(T) = \sum_{n=0}^{N}f_{n}(v_{i}) T^{n}\quad \mathrm{and}\quad Q_{v_{i}}(T) = \sum_{n=0}^{N}g_{n}(v_{i}) T^{n},\]
we write $P_{v_{i}}(T) \sim Q_{v_{i}}(T)$ if $f_{n}(v_{i}) \sim g_{n}(v_{i})$ for all $n$. Note that if $f(v_{i}) \sim g(v_{i})$ holds even along the path $1/\gamma_{i} \coloneqq \{1/ \gamma_{i}(t) \mid t \in [0,1]\}$, then $f(1/v_{i}) \sim g(1/v_{i})$, since $O((1-1/v_{i})^{1-\varepsilon}) = O((1-v_{i})^{1-\varepsilon})$. Additionary, if polynomials $P_{v_{i}}$ and $Q_{v_{i}}$ satisfy either
\[P_{v_{i}}(-\log(1-v_{i})) = Q_{v_{i}}(-\log(1-v_{i}))\] 
or
\[P_{v_{i}}(-\log(1-1/v_{i})) = Q_{v_{i}}(-\log(1-1/v_{i})),\]
and if there exist polynomials $P(T)$ and $Q(T)$ with coefficients independent of $v_{i}$ such that $P_{v_{i}}(T) \sim P(T)$ and $Q_{v_{i}}(T) \sim Q(T)$, then it follows that $P_{v_{i}}(T) \sim Q_{v_{i}}(T)$. This can be verified from
\[-\log(1-v_{i})\, O\left((1-v_{i})^{1-\varepsilon}\right) = O\left((1-v_{i})^{1-\varepsilon}\right).\]
Therefore, in particular, we have
\[\lim_{v_{i} \to 1}P_{v_{i}}(0) = \lim_{v_{i} \to 1}Q_{v_{i}}(0).\]

The following lemma is a consequence of Lemma \ref{le:polynomialz}.
\begin{lemma} For $(\boldsymbol{z}_{1,d-1}, {v}_{d}) \in D_{\boldsymbol{z}, \boldsymbol{\gamma}, d}$ with $z_{d} = 1$, we have
\begin{align}
\mathcal{L}^{+}_{\mathbf{k};\boldsymbol{z}_{1,d-1}, {v}_{d}}(T) &\sim \mathcal{L}^{\shuffle}_{\mathbf{k};\boldsymbol{z}}(T), \label{eq:L+}\\
\mathcal{L}^{-}_{\mathbf{k};\boldsymbol{z}_{1,d-1}, {v}_{d}}(T) &\sim \mathcal{L}^{\shuffle}_{\mathbf{k};\boldsymbol{z}}(T-\log(-z_{d})), \label{eq:L-}\\
\mathcal{L}^{\star, +}_{\mathbf{k};\boldsymbol{z}_{1,d-1}, {v}_{d}}(T) &\sim \mathcal{L}^{\star, \shuffle}_{\mathbf{k};\boldsymbol{z}}(T), \label{eq:Ls+}\\
\mathcal{L}^{\star, -}_{\mathbf{k};\boldsymbol{z}_{1,d-1}, {v}_{d}}(T) &\sim \mathcal{L}^{\star, \shuffle}_{\mathbf{k};\boldsymbol{z}}(T-\log(-z_{d})). \label{eq:Ls-}
\end{align}
\end{lemma}
\begin{proof}
Equation \eqref{eq:L+} follows immediately from Lemma \ref{le:polynomialz}. Equation \eqref{eq:L-} can be verified using
\[\log(-{v}_{d}) = \log(-{z}_{d}) + O(z_{d}-v_{d}).\]
Equations \eqref{eq:Ls+} and \eqref{eq:Ls-} follow from \eqref{eq:v_{d}-z_{d}} and 
\[O\left((z_{h,d}-z_{h,d-1}v_{d})^{1-\varepsilon}\right) = O\left((z_{d} - v_{d})^{1-\varepsilon}\right)\]
as $v_{d} \to z_{d}$ for any $1 \le h \le d$.
\end{proof}
\begin{lemma}\label{le:sim} For $(\boldsymbol{z}_{1,i-1},\boldsymbol{v}_{i,d}) \in D_{\boldsymbol{z}, \boldsymbol{\gamma}, i}$ with $z_{i} = 1$, we have
\begin{align}
\mathcal{B}_{l}({z}_{1,i-1}{v}_{i,d}) &\sim \mathcal{B}_{l}({z}_{1,i}{v}_{i+1,d}), \label{eq:B(v)}\\
\Li_{\mathbf{k}}(\boldsymbol{z}_{1,i-1},\boldsymbol{v}_{i,d}) &\sim \Li_{\mathbf{k}}(\boldsymbol{z}_{1,i},\boldsymbol{v}_{i+1,d}) \quad (i < d\ \mathrm{or}\ k_{d} \neq 1), \label{eq:Li(v)}\\
\Li^{\star}_{\mathbf{k}}(\boldsymbol{z}_{1,i-1},\boldsymbol{v}_{i,d}) &\sim \Li^{\star}_{\mathbf{k}}(\boldsymbol{z}_{1,i},\boldsymbol{v}_{i+1,d}) \quad (i < d\ \mathrm{or}\ k_{d} \neq 1) \label{eq:Lis(v)}.
\end{align}
\end{lemma}
\begin{proof}
Equation \eqref{eq:B(v)} can be verified from
\[\log(-{z}_{1,i-1}{v}_{i,d}) = \log(-{z}_{1,i}{v}_{i+1,d}) + O(z_{i}-v_{i}).\]
We now proceed to prove equation \eqref{eq:Li(v)}. Since the case $i = d$ and $k_{d} = 1$ follows from \eqref{eq:v_{d}-z_{d}}, it is sufficient to consider the case $i < d$. For any $(\boldsymbol{z}_{1,i-1},\boldsymbol{v}_{i,d}) \in D_{\boldsymbol{z}, \boldsymbol{\gamma}, i}$ with $i < d$ (including the case $z_{i} \neq 1$), we prove that
\begin{equation}
\Li_{\mathbf{k}}(\boldsymbol{z}_{1,i-1}, \boldsymbol{v}_{i, d}) - \Li_{\mathbf{k}}(\boldsymbol{z}_{1,i}, \boldsymbol{v}_{i+1,d}) = O\left((z_{i}-v_{i})\log^{i}(z_{i}-v_{i})\right). \label{eq:v_{i}-z_{i}}
\end{equation}
The left-hand side of \eqref{eq:v_{i}-z_{i}} can be written as 
\begin{align*}
&(-1)^{d-i}\int_{0}^{1} \frac{ \Li_{\mathbf{k}_{1,i}}(\boldsymbol{z}_{1,i-1},{v}_{i,d}t)\, dt}{t - 1/v_{i+1,d}} \omega_{0}^{k_{i+1}-1}\cdots \omega_{1/v_{d,d}}\omega_{0} ^{k_{d}-1}\\
&\quad - (-1)^{d-i}\int_{0}^{1} \frac{ \Li_{\mathbf{k}_{1,i}}(\boldsymbol{z}_{1,i-1},z_{i}{v}_{i+1,d}t)\, dt}{t - 1/v_{i+1,d}} \omega_{0}^{k_{i+1}-1}\cdots \omega_{1/v_{d,d}}\omega_{0} ^{k_{d}-1},
\end{align*}
where we recall \eqref{eq:Liint} for the notation. By equation \eqref{eq:v_{d}-z_{d}}, we can estimate
\begin{align*}
&\Li_{\mathbf{k}_{1,i}}(\boldsymbol{z}_{1,i-1},{v}_{i,d}t)-\Li_{\mathbf{k}_{1,i}}(\boldsymbol{z}_{1,i-1},{z}_{i}{v}_{i+1,d}t) = O\left( (z_{i} -v_{i})t\log^{i}((z_{i}-v_{i})t)\right)
\end{align*}
as $v_{i} \to z_{i}$. Since the iterated integral
\[(-1)^{d-i}\int_{0}^{1} \frac{ t\log^{n}{t}\, dt}{t - 1/v_{i+1,d}} \omega_{0}^{k_{i+1}-1}\cdots \omega_{1/v_{d,d}}\omega_{0} ^{k_{d}-1}\]
converges absolutely for any $n \ge 0$, we can obtain equation \eqref{eq:v_{i}-z_{i}}, which implies equation \eqref{eq:Li(v)}. Equation \eqref{eq:Lis(v)} can be understood from \eqref{eq:v_{d}-z_{d}}, \eqref{eq:v_{i}-z_{i}} and
\[O\left((z_{h,i}v_{i+1,j}-z_{h,i-1}v_{i,j})^{1-\varepsilon}\right) = O\left((z_{i} -v_{i})^{1-\varepsilon}\right)\]
as $v_{i} \to z_{i}$, for any $1 \le h \le i \le j \le d$. This completes the proof for all equations.
\end{proof}
The following lemma will be used in computing limits as $v_{i} \to 1$.
\begin{lemma} \label{le:zto1}
For $(\boldsymbol{z}_{1,d-1}, {v}_{d}) \in D_{\boldsymbol{z}, \boldsymbol{\gamma}, d}$ with $z_{d} = 1$, we have
\begin{align}
&\lim_{v_{d} \to 1}\mathcal{L}^{+}_{\mathbf{k};\boldsymbol{z}_{1,d-1}, {v}_{d}}(0) = \Li^{\shuffle}_{\mathbf{k}}(\boldsymbol{z}), \label{eq:L+0} \\
&\lim_{v_{d} \to 1}\mathcal{L}^{\star, +}_{\mathbf{k};\boldsymbol{z}_{1,d-1}, {v}_{d}}(0) = \Li^{\star, \shuffle}_{\mathbf{k}}(\boldsymbol{z}), \label{eq:Ls+0} \\
&\lim_{v_{d} \to 1}\mathcal{L}^{\star, -}_{\mathbf{k};\boldsymbol{z}_{1,d-1}, {v}_{d}}(0) = \sum_{m=0}^{d} (-1)^{d-m}\Li^{\star, \shuffle}_{\mathbf{k}_{1,m}}(\boldsymbol{z}_{1,m})\, \delta_{\mathbf{k}_{m+1,d}}(\boldsymbol{z}_{m+1,d}), \label{eq:Ls-0}\\
&\lim_{v_{d} \to 1}\rho^{-1} \circ \mathcal{L}^{+}_{\mathbf{k};\boldsymbol{z}_{1,d-1}, {v}_{d}}(0) = \Li^{*}_{\mathbf{k}}(\boldsymbol{z}), \label{eq:rL+0}\\
&\lim_{v_{d} \to 1}\rho^{-1} \circ \mathcal{L}^{\star, +}_{\mathbf{k};\boldsymbol{z}_{1,d-1}, {v}_{d}}(0) = \Li^{\star, *}_{\mathbf{k}}(\boldsymbol{z}), \label{eq:rLs+0}\\
&\lim_{v_{d} \to 1} \rho^{-1} \circ \mathcal{L}^{\star, -}_{\mathbf{k};\boldsymbol{z}_{1,d-1}, {v}_{d}}(0) = \sum_{m=0}^{d} (-1)^{d-m}\Li^{\star, *}_{\mathbf{k}_{1,m}}(\boldsymbol{z}_{1,m})\, \delta_{\mathbf{k}_{m+1,d}}(\boldsymbol{z}_{m+1,d}).\label{eq:rLs-0}
\end{align}
\end{lemma}
\begin{proof}
Equations \eqref{eq:L+0} and \eqref{eq:Ls+0} are obvious from equations \eqref{eq:L+} and \eqref{eq:Ls+}, respectively. We proceed to the proof of equation \eqref{eq:Ls-0}. From equations \eqref{eq:L-} and \eqref{eq:shufflepolynomial}, we have
\begin{align}
\lim_{v_{d} \to 1} \mathcal{L}^{-}_{\mathbf{k};\boldsymbol{z}_{1,d-1}, {v}_{d}}(0) &= \mathcal{L}^{\shuffle}_{\mathbf{k};\boldsymbol{z}}(-\log(-z_{d})) \label{eq:L-0=}\\ 
&= \sum_{m=0}^{d} (-1)^{m} \Li^{\shuffle}_{\mathbf{k}_{1,d-m}}(\boldsymbol{z}_{1,d-m})\,\delta_{\mathbf{k}_{d-m+1,d}}(\boldsymbol{z}_{d-m+1,d}). \nonumber
\end{align}
Therefore, we have
\begin{align*}
\lim_{v_{d} \to 1} \mathcal{L}^{\star, -}_{\mathbf{k};\boldsymbol{z}_{1,d-1}, {v}_{d}}(0)
&= \sum_{(\mathbf{k}'; \boldsymbol{z}') \preceq (\mathbf{k}; \boldsymbol{z})} \sum_{m=0}^{d'} (-1)^{m} \Li^{\shuffle}_{\mathbf{k}'_{1, d'-m}}\left(\boldsymbol{z}'_{1, d'-m}\right)\\
&\qquad \times\delta_{\mathbf{k}'_{d'-m+1, d'}}\left(\boldsymbol{z}'_{d'-m+1, d'}\right),
\end{align*}
where $d' \coloneqq \mathrm{dep}(\mathbf{k}')$. The index $\mathbf{k}'_{d'-m+1, d'}$ is written in the form $(k_{n}\boxempty_{n} \cdots \boxempty_{d-1}k_{d})$. When at least one $\boxempty_{i}$ is $\text{``} + \text{''}$ for $i \in \{n,\dots,d -1 \}$, 
\begin{equation}\label{eq:delta}
\delta_{\mathbf{k}'_{d'-m+1, d'}}\left(\boldsymbol{z}'_{d'-m+1, d'}\right)
\end{equation}
vanishes. Therefore, the factor \eqref{eq:delta} can be nonzero only if
\[\mathbf{k}'_{d'-m+1, d'} = \mathbf{k}_{d-m+1, d} \text{\quad and \quad } \boldsymbol{z}'_{d'-m+1, d'} = \boldsymbol{z}_{d-m+1,d},\]
in which case $\mathbf{k}'_{1, d'-m}$ is written in the form $(k_{1}\boxempty_{1} \cdots \boxempty_{d-m} k_{d-m})$.
Thus, we obtain
\[\lim_{v_{d} \to 1} \mathcal{L}^{\star, -}_{\mathbf{k};\boldsymbol{z}_{1,d-1}, {v}_{d}}(0) = \sum_{m=0}^{d} (-1)^{m}\Li^{\star, \shuffle}_{\mathbf{k}_{1,d-m}}(\boldsymbol{z}_{1,d-m})\, \delta_{\mathbf{k}_{d-m+1,d}}(\boldsymbol{z}_{d-m+1,d}),\]
which corresponds to equation \eqref{eq:Ls-0}. We proceed to the proof of equation \eqref{eq:rL+0}.
By applying $\rho^{-1}$ to both sides of equation \eqref{eq:L+}, we obtain
\[\rho^{-1} \circ \mathcal{L}^{+}_{\mathbf{k};\boldsymbol{z}_{1,d-1}, {v}_{d}}(T) \sim \rho^{-1} \circ \mathcal{L}^{\shuffle}_{\mathbf{k};\boldsymbol{z}}(T).\]
This implies that
\[\lim_{v_{d} \to 1} \rho^{-1} \circ \mathcal{L}^{+}_{\mathbf{k};\boldsymbol{z}_{1,d-1}, {v}_{d}}(0) = \rho^{-1} \circ \mathcal{L}^{\shuffle}_{\mathbf{k};\boldsymbol{z}}(0) = \Li^{*}_{\mathbf{k}}(\boldsymbol{z}),\]
and therefore equation \eqref{eq:rL+0} holds. Equation \eqref{eq:rLs+0} is obtained by summing the above equation. We proceed to the proof of equation \eqref{eq:rLs-0}. We have
\begin{align*}
\mathcal{L}^{\shuffle}_{\mathbf{k};\boldsymbol{z}}(T-\log(-z_{d})) 
&=\sum_{m=0}^{h} \frac{1}{m!} \Li^{\shuffle}_{\mathbf{k}_{1,d-m}}(\boldsymbol{z}_{1,d-m})\, (T-\log(-z_{d}))^{m}\\
&= \sum_{j=0}^{h} \frac{(-\log(-z_{d}))^{j}}{j!} \sum_{m=j}^{h} \Li^{\shuffle}_{\mathbf{k}_{1,d-m}}(\boldsymbol{z}_{1,d-m})\frac{T^{m-j}}{(m-j)!} \\
&= \sum_{j=0}^{h} \mathcal{L}^{\shuffle}_{\mathbf{k}_{1,d-j};\boldsymbol{z}_{1,d-j}}(T)\, \frac{(-\log(-z_{d}))^{j}}{j!},
\end{align*}
where $h$ has the same meaning as in \eqref{eq:shufflepolynomial}. Therefore, by applying $\rho^{-1}$ to both sides of equation \eqref{eq:L-}, we have
\[\rho^{-1} \circ \mathcal{L}^{-}_{\mathbf{k};\boldsymbol{z}_{1,d-1}, {v}_{d}}(T) \sim \sum_{m=0}^{d} (-1)^{m}\mathcal{L}^{*}_{\mathbf{k}_{1,d-m};\boldsymbol{z}_{1,d-m}}(T)\,\delta_{\mathbf{k}_{d-m+1,d}}(\boldsymbol{z}_{d-m+1,d}),\]
which implies that
\[\lim_{v_{d} \to 1} \rho^{-1} \circ \mathcal{L}^{-}_{\mathbf{k};\boldsymbol{z}_{1,d-1}, {v}_{d}}(0) = \sum_{m=0}^{d}(-1)^{m} \Li^{*}_{\mathbf{k}_{1,d-m}}(\boldsymbol{z}_{1,d-m}) \, \delta_{\mathbf{k}_{d-m+1,d}}(\boldsymbol{z}_{d-m+1,d}).\]
Using the same method as was used to derive equation \eqref{eq:Ls-0} from \eqref{eq:L-0=}, we can obtain equation \eqref{eq:rLs-0}. This completes the proof for all equations.
\end{proof}

\subsection{Proof of Theorem \ref{th:main2}}\label{subsec:Proof2}
In this subsection, we complete the proof of Theorem \ref{th:main2}. As explained in Section \ref{subsec:outline2}, we fix $\boldsymbol{z} \in \mathcal{D}^{d}(\mathbb{R}_{\ge 0}\setminus\{1\})$, and compute the limit of equation \eqref{eq:P=Q} along the paths $\gamma_{1}, \dots, \gamma_{d}$, taking the limit sequentially in the order $v_{1} \to z_{1}, \dots, v_{d} \to z_{d}$.

For convenience, we define
\begin{align*}
Q^{\bullet}_{\mathbf{k}}(\boldsymbol{z}) &\coloneqq \sum_{0 \le m < n \le d} \sum_{\substack{a+b+ l=k_{n} \\ a,b, l \ge 0}}  (-1)^{m+|\mathbf{k}_{n+1,d}|+b} \mathcal{B}_{l}(z_{m+1,d})\\ \nonumber
&\qquad \times \Li^{\star, \bullet}_{\mathbf{k}_{1,m}}(\boldsymbol{z}_{1,m}) \tilde{\Li}^{a, \bullet}_{\mathbf{k}_{m+1,n-1}}(\boldsymbol{z}_{m+1,n-1})\Li^{b}_{\mathbf{k}_{n+1,d}}(1/\boldsymbol{z}_{n+1,d})
\end{align*}
for $\bullet = \{\shuffle , *\}$. We prove Theorem \ref{th:main2} separately for the case $\bullet = \shuffle$ and the case $\bullet = *$.

\textbf{Case $\bullet = \shuffle$.} We first compute the limit of equation \eqref{eq:P=Q} as $v_{1} \to z_{1}$ along the path $\gamma_{1}$. We assume $1 < d$. If $(k_{1},z_{1}) \neq (1,1)$, then neither side of equation \eqref{eq:P=Q} involves any divergent factors. Hence, we obtain
\[P_{\mathbf{k}}({z}_{1},\boldsymbol{v}_{2,d}) = Q_{\mathbf{k}}({z}_{1},\boldsymbol{v}_{2,d}) = Q^{\shuffle}_{\mathbf{k}}({z}_{1},\boldsymbol{v}_{2,d}).\]
When $(k_{1},z_{1}) = (1,1)$, the left-hand side of equation \eqref{eq:P=Q} does not involve any divergent factors, and we have $P_{\mathbf{k}}(\boldsymbol{v}) \sim P_{\mathbf{k}}({z}_{1},\boldsymbol{v}_{2,d})$. The right-hand side of equation \eqref{eq:P=Q} involves the polylogarithm $\Li^{\star}_{1}({v}_{1})$, which is the only divergent factor. Let $Q_{\mathbf{k};\boldsymbol{v}}(T)$ be the polynomial obtained by replacing each $\Li^{\star}_{1}({v}_{1})$ in $Q_{\mathbf{k}}(\boldsymbol{v})$ with $\mathcal{L}^{\star, +}_{{1};{v}_{1}}(T)$. Then, $P_{\mathbf{k}}(\boldsymbol{v}) = Q_{\mathbf{k}}(\boldsymbol{v}) = Q_{\mathbf{k};\boldsymbol{v}}(-\log(1-v_{1}))$ holds, and, by equation \eqref{eq:Ls+} and Lemma \ref{le:sim}, there exists a polynomial $Q(T)$ with coefficients independent of $v_{1}$ such that $Q_{\mathbf{k};\boldsymbol{v}}(T) \sim Q(T)$. Therefore, we have
\begin{equation}
P_{\mathbf{k}}(\boldsymbol{v}) \sim Q_{\mathbf{k};\boldsymbol{v}}(T). \label{eq:PequivQT}
\end{equation}
In particular, using equation \eqref{eq:Ls+0}, we obtain
\[P_{\mathbf{k}}({z}_{1},\boldsymbol{v}_{2,d}) = \lim_{v_{1} \to 1}Q_{\mathbf{k};\boldsymbol{v}}(0) = Q^{\shuffle}_{\mathbf{k}}({z}_{1},\boldsymbol{v}_{2,d}).\]
Therefore, regardless of the value of $z_{1}$, we conclude that
\begin{equation}
P_{\mathbf{k}}({z}_{1},\boldsymbol{v}_{2,d}) = Q^{\shuffle}_{\mathbf{k}}({z}_{1},\boldsymbol{v}_{2,d}). \label{eq:shufflez_{1}}
\end{equation}

Next, we compute the limit of equation \eqref{eq:shufflez_{1}} as $v_{2} \to z_{2}$ along the path $\gamma_{2}$. We assume $2 < d$. If $(k_{2},z_{2}) \neq (1,1)$, then neither side involves any divergent factors, and hence we have
\begin{align}
P_{\mathbf{k}}(\boldsymbol{z}_{1,2},\boldsymbol{v}_{3,d}) = Q^{\shuffle}_{\mathbf{k}}(\boldsymbol{z}_{1,2},\boldsymbol{v}_{3,d}). \label{eq:shufflez_{2}}
\end{align}
When $(k_{2},z_{2}) = (1,1)$, the left-hand side of equation \eqref{eq:shufflez_{1}} does not involve any divergent factors, and we have $P_{\mathbf{k}}({z}_{1},\boldsymbol{v}_{2,d}) \sim P_{\mathbf{k}}(\boldsymbol{z}_{1,2},\boldsymbol{v}_{3,d})$. The right-hand side involves the polylogarithm $\Li^{\star}_{1}({v}_{2})$ and multiple polylogarithms of the form $\Li^{\star}_{k',1}(z_{1},{v}_{2})$, which are the only divergent factors. Let $Q^{\shuffle}_{\mathbf{k};{z}_{1},\boldsymbol{v}_{2,d}}(T)$ be the polynomial obtained by replacing each $\Li^{\star}_{1}({v}_{2})$ and $\Li^{\star}_{k',1}(z_{1},{v}_{2})$ in $Q^{\shuffle}_{\mathbf{k}}({z}_{1},\boldsymbol{v}_{2,d})$ with $\mathcal{L}^{\star, +}_{{1};{v}_{2}}(T)$ and  $\mathcal{L}^{\star, +}_{k',1;z_{1},{v}_{2}}(T)$, respectively. Then, it follows that
\[P_{\mathbf{k}}({z}_{1},\boldsymbol{v}_{2,d}) \sim Q^{\shuffle}_{\mathbf{k};{z}_{1},\boldsymbol{v}_{2,d}}(T),\]
and in particular,
\[P_{\mathbf{k}}(\boldsymbol{z}_{1,2},\boldsymbol{v}_{3,d}) = \lim_{v_{2} \to 1}Q^{\shuffle}_{\mathbf{k};{z}_{1},\boldsymbol{v}_{2,d}}(0) = Q^{\shuffle}_{\mathbf{k}}(\boldsymbol{z}_{1,2},\boldsymbol{v}_{3,d}),\]
which coincides with equation \eqref{eq:shufflez_{2}}.

By continuing this process up to $v_{d-1} \to z_{d-1}$, we conclude that
\begin{equation}
P_{\mathbf{k}}(\boldsymbol{z}_{1,d-1}, {v}_{d}) = Q^{\shuffle}_{\mathbf{k}}(\boldsymbol{z}_{1,d-1}, {v}_{d}). \label{eq:shufflez_{d-1}}
\end{equation}
As the final step, we compute the limit of equation \eqref{eq:shufflez_{d-1}} as $v_{d} \to z_{d}$ along the path $\gamma_{d}$. If $(k_{d},z_{d}) \neq (1,1)$, then neither side involves any divergent factors, and hence we obtain $P_{\mathbf{k}}(\boldsymbol{z}) = Q^{\shuffle}_{\mathbf{k}}(\boldsymbol{z})$, which corresponds to equation \eqref{eq:main2} for $\bullet = \shuffle$ and $(k_{d},z_{d}) \neq (1,1)$. When $(k_{d},z_{d}) = (1,1)$, both sides of equation $\eqref{eq:shufflez_{d-1}}$ diverge as $v_{d} \to z_{d}$. Let 
\begin{equation}
P_{\mathbf{k};\boldsymbol{z}_{1,d-1}, {v}_{d}}(T) \coloneqq (-1)^{d}\mathcal{L}^{\star, -}_{\mathbf{k};\boldsymbol{z}_{1,d-1}, {v}_{d}}(T) - (-1)^{|\mathbf{k}|}\mathcal{L}^{+}_{\mathbf{k}; 1/\boldsymbol{z}_{1,d-1}, 1/{v}_{d}}(T) \label{eq:defP}
\end{equation}
and let $Q^{\shuffle}_{\mathbf{k};\boldsymbol{z}_{1,d-1}, {v}_{d}}(T)$ be the polynomial obtained by replacing multiple polylogarithms of the form $\Li^{}_{\mathbf{k}',1}(1/\boldsymbol{z}', 1/{v}_{d})$  in $Q^{\shuffle}_{\mathbf{k}}(\boldsymbol{z}_{1,d-1}, {v}_{d})$ with $\mathcal{L}^{+}_{\mathbf{k}', 1; 1/\boldsymbol{z}', 1/{v}_{d}}(T)$. Then, equation \eqref{eq:shufflez_{d-1}} can be rewritten as
\[P_{\mathbf{k};\boldsymbol{z}_{1,d-1}, {v}_{d}}(-\log(1-1/v_{d})) = Q^{\shuffle}_{\mathbf{k};\boldsymbol{z}_{1,d-1}, {v}_{d}}(-\log(1-1/v_{d})).\]
In addition, by equations \eqref{eq:L+} and \eqref{eq:Ls-} and Lemma \ref{le:sim}, there exist polynomials $P(T)$ and $Q(T)$ with coefficients independent of $v_{i}$ such that $P_{\mathbf{k};\boldsymbol{z}_{1,d-1}, {v}_{d}}(T) \sim P(T)$ and $Q^{\shuffle}_{\mathbf{k};\boldsymbol{z}_{1,d-1}, {v}_{d}}(T) \sim Q(T)$. Therefore, we have
\[P_{\mathbf{k};\boldsymbol{z}_{1,d-1}, {v}_{d}}(T) \sim Q^{\shuffle}_{\mathbf{k};\boldsymbol{z}_{1,d-1}, {v}_{d}}(T),\]
and in particular,
\[\lim_{v_{d} \to 1} P_{\mathbf{k};\boldsymbol{z}_{1,d-1}, {v}_{d}}(0) = \lim_{v_{d} \to 1}  Q^{\shuffle}_{\mathbf{k};\boldsymbol{z}_{1,d-1}, {v}_{d}}(0).\]
By applying equations \eqref{eq:L+0} and  \eqref{eq:Ls-0} to the above equation, we obtain equation \eqref{eq:main2} for $\bullet = \shuffle$ and $(k_{d},z_{d}) = (1,1)$. This completes the proof of Theorem \ref{th:main2} for $\bullet = \shuffle$.

\textbf{Case $\bullet = *$.}

The approach is the same as in the case $\bullet = \shuffle $, except for the use of $\rho^{-1}$. We first compute the limit of equation \eqref{eq:P=Q} as $v_{1} \to z_{1}$ along the path $\gamma_{1}$. We assume $1 < d$. If $(k_{1},z_{1}) \neq (1,1)$, then we have
\[P_{\mathbf{k}}({z}_{1},\boldsymbol{v}_{2,d}) = Q_{\mathbf{k}}({z}_{1},\boldsymbol{v}_{2,d}) = Q^{*}_{\mathbf{k}}({z}_{1},\boldsymbol{v}_{2,d}).\]
When $(k_{1},z_{1}) = (1,1)$, as in the case $\bullet = \shuffle$, we have equation \eqref{eq:PequivQT}. By applying $\rho^{-1}$ to both sides of equation \eqref{eq:PequivQT}, we obtain
\[P_{\mathbf{k}}(\boldsymbol{v}) \sim \rho^{-1} \circ Q_{\mathbf{k};\boldsymbol{v}}(T).\]
Using equation \eqref{eq:rLs+0}, we obtain
\[P_{\mathbf{k}}({z}_{1},\boldsymbol{v}_{2,d}) = \lim_{v_{1} \to 1}\rho^{-1} \circ Q_{\mathbf{k};\boldsymbol{v}}(0) = Q^{*}_{\mathbf{k}}({z}_{1},\boldsymbol{v}_{2,d}).\]
Therefore, regardless of the value of $z_{1}$, it follows that
\begin{equation}
P_{\mathbf{k}}({z}_{1},\boldsymbol{v}_{2,d}) = Q^{*}_{\mathbf{k}}({z}_{1},\boldsymbol{v}_{2,d}). \label{eq:stufflez_{1}}
\end{equation}

Next, we compute the limit of equation \eqref{eq:stufflez_{1}} as $v_{2} \to z_{2}$ along the path $\gamma_{2}$. We assume $2 < d$. If $(k_{2},z_{2}) \neq (1,1)$, then we have
\begin{equation}
P_{\mathbf{k}}(\boldsymbol{z}_{1,2},\boldsymbol{v}_{3,d}) = Q^{*}_{\mathbf{k}}(\boldsymbol{z}_{1,2},\boldsymbol{v}_{3,d}). \label{eq:stufflez_{2}}
\end{equation}
We consider the case $(k_{2},z_{2}) = (1,1)$. Let $Q^{*}_{\mathbf{k};{z}_{1},\boldsymbol{v}_{2,d}}(T)$ be the polynomial obtained by replacing each $\Li^{\star}_{1}({v}_{2})$ and $\Li^{\star}_{k',1}(z_{1},{v}_{2})$ in $Q^{*}_{\mathbf{k}}({z}_{1},\boldsymbol{v}_{2,d})$ with $\mathcal{L}^{\star, +}_{{1};{v}_{2}}(T)$ and  $\mathcal{L}^{\star, +}_{k',1;z_{1},{v}_{2}}(T)$, respectively. Because $P_{\mathbf{k}}({z}_{1},\boldsymbol{v}_{2,d}) \sim Q^{*}_{\mathbf{k};{z}_{1},\boldsymbol{v}_{2,d}}(T)$ holds, we obtain
\[P_{\mathbf{k}}({z}_{1},\boldsymbol{v}_{2,d}) \sim \rho^{-1} \circ Q^{*}_{\mathbf{k};{z}_{1},\boldsymbol{v}_{2,d}}(T).\]
In particular, it follows that
\[P_{\mathbf{k}}(\boldsymbol{z}_{1,2},\boldsymbol{v}_{3,d}) = \lim_{v_{2} \to 1}\rho^{-1} \circ Q^{*}_{\mathbf{k};{z}_{1},\boldsymbol{v}_{2,d}}(0) = Q^{*}_{\mathbf{k}}(\boldsymbol{z}_{1,2},\boldsymbol{v}_{3,d}),\]
which coincides with equation \eqref{eq:stufflez_{2}}. 

By continuing this process up to $v_{d-1} \to z_{d-1}$, we conclude that
\begin{equation}
P_{\mathbf{k}}(\boldsymbol{z}_{1,d-1}, {v}_{d}) = Q^{*}_{\mathbf{k}}(\boldsymbol{z}_{1,d-1}, {v}_{d}). \label{eq:stufflez_{d-1}}
\end{equation}

As the final step, we compute the limit of equation \eqref{eq:stufflez_{d-1}} as $v_{d} \to z_{d}$ along the path $\gamma_{d}$. If $(k_{d},z_{d}) \neq (1,1)$, then neither side involves any divergent factors, and hence we obtain $P_{\mathbf{k}}(\boldsymbol{z}) = Q^{*}_{\mathbf{k}}(\boldsymbol{z})$, which corresponds to equation \eqref{eq:main2} for $\bullet = *$ and $(k_{d},z_{d}) \neq (1,1)$. When $(k_{d},z_{d}) = (1,1)$, both sides of equation $\eqref{eq:stufflez_{d-1}}$ diverge as $v_{d} \to z_{d}$. Let $Q^{*}_{\mathbf{k};\boldsymbol{z}_{1,d-1}, {v}_{d}}(T)$ be the polynomial obtained by replacing multiple polylogarithms of the form $\Li^{}_{\mathbf{k}',1}(1/\boldsymbol{z}', 1/{v}_{d})$  in $Q^{*}_{\mathbf{k}}(\boldsymbol{z}_{1,d-1}, {v}_{d})$ with $\mathcal{L}^{+}_{\mathbf{k}', 1; 1/\boldsymbol{z}', 1/{v}_{d}}(T)$. 
Then, it follows that
\[P_{\mathbf{k};\boldsymbol{z}_{1,d-1}, {v}_{d}}(T) \sim Q^{*}_{\mathbf{k};\boldsymbol{z}_{1,d-1}, {v}_{d}}(T),\]
where $P_{\mathbf{k};\boldsymbol{z}_{1,d-1}, {v}_{d}}(T)$ is defined in equation \eqref{eq:defP}. Therefore, we have
\[\rho^{-1} \circ P^{}_{\mathbf{k};\boldsymbol{z}_{1,d-1}, {v}_{d}}(T) \sim \rho^{-1} \circ Q^{*}_{\mathbf{k};\boldsymbol{z}_{1,d-1}, {v}_{d}}(T),\]
and, in particular,
\[\lim_{v_{d} \to 1} \rho^{-1} \circ P^{}_{\mathbf{k};\boldsymbol{z}_{1,d-1}, {v}_{d}}(0) = \lim_{v_{d} \to 1}  \rho^{-1} \circ Q^{*}_{\mathbf{k};\boldsymbol{z}_{1,d-1}, {v}_{d}}(0).\]
By applying equations \eqref{eq:rL+0} and \eqref{eq:rLs-0} to the above equation, we obtain equation \eqref{eq:main2} for $\bullet = *$ and $(k_{d},z_{d}) = (1,1)$.

This completes the proof of Theorem \ref{th:main2}.

\section*{Acknowledgment}
The author is deeply grateful to Prof.\ Minoru Hirose, Prof.\ Yayoi Nakamura and Prof.\ Kentaro Ihara for their helpful comments.

\vspace{4mm}

\end{document}